\documentclass{amsart}

\usepackage{amssymb}
\usepackage{amsfonts}
\usepackage{amsmath}
\usepackage{amsthm}
\usepackage{graphicx}
\usepackage{mathrsfs}
\usepackage{dsfont}
\usepackage{amscd}
\usepackage{multirow}
\usepackage{mathpazo}
\usepackage[all]{xy}
\usepackage[scaled]{helvet} 
\usepackage{courier} 
\normalfont
\usepackage[T1]{fontenc}
\usepackage{calligra}
\usepackage{verbatim}
\usepackage[usenames,dvipsnames]{color}
\usepackage[colorlinks=true,linkcolor=Blue,citecolor=Violet]{hyperref}

\makeatletter
\@namedef{subjclassname@2010}{%
  \textup{2010} Mathematics Subject Classification}
\makeatother

\theoremstyle{plain}
\newtheorem{theorem}{Theorem}[section]

\newtheorem{claim}[theorem]{Claim}

\newtheorem{case}{Case}

\newtheorem{corollary}[theorem]{Corollary}

\newtheorem{lemma}[theorem]{Lemma}
\newtheorem{proposition}[theorem]{Proposition}

\newtheorem{theorem-definition}[theorem]{Theorem-Definition}

\theoremstyle{definition}
\newtheorem{definition}[theorem]{Definition}

\newtheorem{notation}[theorem]{Notation}

\newtheorem{convention}[theorem]{Convention}

\theoremstyle{remark}

\newtheorem{example}[theorem]{Example}

\newtheorem{remark}[theorem]{Remark}

\numberwithin{equation}{section}

\newcommand{\N}{{\mathds{N}}}

\newcommand{\Q}{{\mathds{Q}}}
\newcommand{\R}{{\mathds{R}}}
\newcommand{\C}{{\mathds{C}}}
\newcommand{\T}{{\mathds{T}}}

\newcommand{\D}{{\mathfrak{D}}}
\newcommand{\A}{{\mathfrak{A}}}
\newcommand{\B}{{\mathfrak{B}}}
\newcommand{\M}{{\mathfrak{M}}}

\newcommand{\Lip}{{\mathsf{L}}}

\newcommand{\qpropinquity}[1]{{\mathsf{\Lambda}_{#1}}}

\newcommand{\Kantorovich}[1]{{\mathsf{mk}_{#1}}}

\newcommand{\Haus}[1]{{\mathsf{Haus}_{#1}}}

\newcommand{\StateSpace}{{\mathscr{S}}}

\newcommand{\mongekant}{{Mon\-ge-Kan\-to\-ro\-vich metric}}

\newcommand{\Lqcms}{{\JLL} quantum compact metric space}
\newcommand{\Qqcms}[1]{${#1}$-quasi-Leibniz quantum compact metric space}
\newcommand{\gQqcms}{quasi-Leibniz quantum compact metric space}

\newcommand{\unit}{1}

\newcommand{\sa}[1]{{\mathfrak{sa}\left({#1}\right)}}

\newcommand{\JLL}{Lei\-bniz}

\newcommand{\dom}[1]{{\operatorname*{dom}({#1})}}
\newcommand{\diam}[2]{{\mathrm{diam}\left({#1},{#2}\right)}}

\newcommand{\Lipball}[2]{{\mathfrak{Lip}_{#1}\left({#2}\right)}}
\newcommand{\bridgereach}[2]{{\varrho\left({#1}\middle|{#2}\right)}}
\newcommand{\bridgeheight}[2]{{\varsigma\left({#1}\middle|{#2}\right)}}

\newcommand{\bridgelength}[2]{{\lambda\left({#1}\middle|{#2}\right)}}

\newcommand{\bridgenorm}[2]{{\mathsf{bn}_{ {#1}  }\left({#2}\right)}}

\newcommand{\alg}[1]{{\mathfrak{#1}}}

\newcommand{\BaireSpace}{{\mathscr{N}}}

\newcommand{\indmor}[2]{{\underrightarrow{#1^{#2}}}}
\newcommand{\af}[1]{{\alg{AF}_{#1}}}

\renewcommand{\geq}{\geqslant}
\renewcommand{\leq}{\leqslant}

\newcommand{\CondExp}[2]{{\mathds{E}\left({#1}\middle\vert{#2}\right)}}

\newcommand{\Latremoliere}{Latr\'{e}moli\`{e}re}

\allowdisplaybreaks[4]

\makeatletter
\newcommand{\vast}{\bBigg@{4}}
\newcommand{\Vast}{\bBigg@{5}}
\makeatother

\begin{document}
\title[AF algebras in the quantum propinquity space]{AF algebras in the quantum Gromov-Hausdorff propinquity space}
\author{Konrad Aguilar}
\address{School of Mathematical and Statistical Sciences \\ Arizona State University \\ Tempe AZ 85281}
\email{konrad.aguilar@gmail.com}


\date{\today}
\thanks{The author  was partially supported by the Simons - Foundation grant 346300 and the Polish Government MNiSW 2015-2019 matching fund.}
\thanks{Some of this work was completed during the Simons semester hosted at IMPAN during September-December 2016 titled ``Noncommutative Geometry - The Next Generation''.}

\begin{abstract}
For unital AF algebras with faithful tracial states, we provide criteria for convergence in the quantum Gromov-Hausdorff propinquity.  For any unital AF algebra, we introduce new Leibniz Lip-norms using quotient norms.  Next, we show that any unital AF algebra is the limit of its defining inductive sequence of finite dimensional C*-algebras in quantum propinquity given any  Lip-norm defined on the inductive sequence.  We also establish sufficient conditions for when a *-isomorphism between two AF algebras produces a quantum isometry in the context of various Lip-norms. 
\end{abstract}
\maketitle
\setcounter{tocdepth}{2}
\tableofcontents


\section{Introduction}

The Gromov-Hausdorff propinquity \cite{Latremoliere13,Latremoliere13b,Latremoliere14,Latremoliere15,  Latremoliere16}, a noncommutative analogue of the Gromov-Hausdorff distance, provides a new framework to study the geometry of classes of C*-algebras, opening new avenues of research in noncommutative geometry by work of \Latremoliere \ building off notions introduced by Rieffel \cite{Rieffel98a, Rieffel00} and Connes \cite{Connes89, Connes}.  In collaboration with \Latremoliere, our previous work in \cite{AL} served to introduce AF algebras into the realm of noncommutaive metric geometry. In particular, given a unital AF algebra with a faithful tracial state, we endowed such an AF algebra, viewed as an inductive limit,  with quantum metric structure and showed that these AF algebras are indeed limits of the given inductive sequence  of finite dimensional algebras in the quantum propinquity topology. From here, we were able to construct a H\"{o}lder-continuous function from the Baire space onto the class of UHF algebras of Glimm \cite{Glimm60} and a continuous map from the Baire space onto the class Effros-Shen AF algebras introduced in \cite{Effros80b} viewed as quantum metric spaces, and therefore both these classes inherit the metric geometry of the Baire Space via continuous maps.

In Section (\ref{conv-af}), we continue this journey and develop some useful generalizations to our work in \cite{AL}.  In particular, we note that the given a unital AF algebra with faithful tracial state, the Lip-norm constructed in \cite{AL} is constructed by three structures:  the inductive sequence,  the faithful tracial state, and a positive sequence vanishing at infinity, which is usually taken to be some form of the reciprocal of the dimensions of the finite dimensional spaces of the inductive sequence.  From this, we provide suitable notions of convergence for all of these three structures, which all together imply convergence of families of AF algebras, in which we introduce the notion of fusing families of inductive limits. This imparts a generalization that implies the continuity results of \cite{AL} pertaining to the UHF and Effros-Shen AF algebras.  The results of this section greatly simplify the task of providing convergence of AF algebras in the quantum propinquity topology and should prove useful in future projects (see \cite{Aguilar16}).

Next, Rieffel's work in \cite{Rieffel11} allows us to produce Leibniz Lip-norms for all unital AF algebras with or without faithful tracial state via quotient norms instead of conditional expectations in Section (\ref{Leibniz-Lip-AF}).  Furthermore, these Lip-norms still provide convergence of finite-dimensional subalgebras in quantum proqintuity to the AF algebra as is done also in the faithful tracial state case. However, in this section, we show a stronger result in that convergence of the finite dimensional subalgebras only depends on the existence of a Lip-norm satsifying a natural density condition on its domain. This shows that all unital AF algebras are also approximately finite dimensional in the sense of propinquity.

Finally, we study quantum isometries between AF algebras or when two AF algebras produce the same equivlance classes that form the quantum propinquity metric space. We split this up between the cases of the Lip-norms from faithful tracial state and the Lip-norms from quotient norms.   Furthermore, since finite dimensional approximations are so crucial to convergence results, the conditions for quantum isometries for AF algebras are also chosen so that they produce quantum isometries for the finite dimensional subalgebras of the inductive sequence.

\section{Background: Quantum Metric Geometry and AF algebras}
The purpose of this section is to discuss our progress thus far in the realm of quantum metric spaces with regard to AF algebras, but we also provide a cursory overview of the material on quantum compact metric spaces since we will utilize each result listed here in some capacity later in the paper.  We refer the reader to the survey by \Latremoliere \ \cite{Latremoliere15b} for a much more detailed and insightful introduction to the study of quantum metric spaces.

\begin{notation}
When $E$ is a normed vector space, then its norm will be denoted by $\|\cdot\|_E$ by default.
\end{notation}

\begin{notation}
Let $\A$ be a unital C*-algebra. The unit of $\A$ will be denoted by $\unit_\A$. The state space of $\A$ will be denoted by $\StateSpace(\A)$ while the self-adjoint part of $\A$ will be denoted by $\sa{\A}$. 
\end{notation}

\begin{definition}[\cite{Rieffel98a, Latremoliere13, Latremoliere15}]\label{quasi-Monge-Kantorovich-def}
A {\Qqcms{(C,D)}} $(\A,\Lip)$, for some $C\geq 1$ and $D\geq 0$, is an ordered pair where $\A$ is unital C*-algebra and $\Lip$ is a seminorm defined $\sa{\A}$ such that $\dom{\Lip}=\{ a \in \sa{\A}: \Lip (a) < \infty \}$ is a unital dense Jordan-Lie subalgebra  of $\sa{\A}$ such that:
\begin{enumerate}
\item $\{ a \in \sa{\A} : \Lip(a) = 0 \} = \R\unit_\A$,
\item the seminorm $\Lip$ is a \emph{$(C,D)$-quasi-Leibniz Lip-norm}, i.e. for all $a,b \in \dom{\Lip}$:
\begin{equation*}
\max\left\{ \Lip\left(\frac{ab+ba}{2}\right), \Lip\left(\frac{ab-ba}{2i}\right) \right\} \leq C\left(\|a\|_\A \Lip(b) + \|b\|_\A\Lip(a)\right) + D \Lip(a)\Lip(b)\text{,}
\end{equation*}
where $\frac{ab+ba}{2}$ and $ \frac{ab-ba}{2i}$ define the Jordan and Lie product, respectively,
\item the \emph{\mongekant} defined, for all two states $\varphi, \psi \in \StateSpace(\A)$, by:
\begin{equation*}
\Kantorovich{\Lip} (\varphi, \psi) = \sup\left\{ |\varphi(a) - \psi(a)| : a\in\dom{\Lip}, \Lip(a) \leq 1 \right\}
\end{equation*}
metrizes the weak* topology of $\StateSpace(\A)$,
\item the seminorm $\Lip$ is lower semi-continuous with respect to $\|\cdot\|_\A$.
\end{enumerate}
If all but (2) above is satisfied, then we call $(\A,\Lip)$ a {\em quantum compact metric space}.
\end{definition}

Rieffel initiated the systematic study of quantum compact metric space with the following characterization of these spaces, which can be seen as a noncommutative form of the Arz{\'e}la-Ascoli theorem.

\begin{theorem}[\cite{Rieffel98a,Rieffel99, Ozawa05}]\label{Rieffel-thm}
Let $\A$ be a unital C*-algebra and $\Lip$ a seminorm defined on $\sa{\A}$ such that $\dom{\Lip}$ is a unital dense subspace of $\sa{\A}$ and $\Lip(a) = 0$ if and only if $a\in\R\unit_\A$. The following three assertions are equivalent:
\begin{enumerate}
\item the {\mongekant} $\Kantorovich{\Lip}$ metrizes the weak* topology on $\StateSpace(\A)$;
\item for some state $\mu \in \StateSpace(\A)$, the set:
\begin{equation*}
 \Lipball{1}{\A, \Lip_\A, \mu }:=\left\{ a\in\sa{\A} : \Lip(a)\leq 1, \mu(a) = 0\right\}
\end{equation*}
is totally bounded for $\|\cdot\|_\A$;
\item 
the set:
\begin{equation*} \{ a+\R1_\A \in \sa{\A}/\R1_\A : a \in \dom{\Lip}, \Lip(a) \leq 1 \}
\end{equation*} is totally bounded in $\sa{\A}/\R1_\A$ for $\Vert \cdot \Vert_{\sa{\A}/\R1_\A}$; 
\end{enumerate}
\end{theorem}
A primary interest in developing a theory of quantum metric spaces is the introduction of various hypertopologies on classes of such spaces, thus allowing us to study the geometry of classes of C*-algebras and perform analysis on these classes. A classical model for our hypertopologies is given by the Gromov-Hausdorff distance \cite{Gromov81, Gromov}. While several noncommutative analogues of the Gromov-Hausdorff distance have been proposed --- most importantly Rieffel's original construction of the quantum Gromov-Hausdorff distance \cite{Rieffel00} --- we shall work with a particular metric introduced by \Latremoliere, \cite{Latremoliere13}, as we did in \cite{AL}. This metric, known as the quantum propinquity, is designed to be best suited to {\gQqcms s}, and in particular, is zero between two such spaces if and only if they are quantum isometric. We now propose a summary of the tools needed to compute upper bounds on this metric.

\begin{definition}[{\cite[Definition 3.1]{Latremoliere13}}]
The \emph{$1$-level set} $\StateSpace_1(\D|\omega)$ of an element $\omega$ of a unital C*-algebra $\D$ is:
\begin{equation*}
\left\{ \varphi \in \StateSpace(\D) : \varphi((1-\omega^\ast\omega))=\varphi((1-\omega \omega^\ast)) = 0 \right\}\text{.}
\end{equation*}
\end{definition}
Next, we define the notion of a  \Latremoliere \ bridge, which is not only crucial in the definition of the quantum propinquity but also the convergence results of \Latremoliere \ in \cite{Latremoliere13c} and Rieffel in \cite{Rieffel15}. In particular, the pivot of Definition (\ref{bridge-def}) is of utmost importance in the convergence results of \cite{Latremoliere13c, Rieffel15}.
\begin{definition}[{\cite[Definition 3.6]{Latremoliere13}}]\label{bridge-def}
A \emph{bridge} from $\A$ to $\B$, where $\A$ and $\B$ are unital C*-algebras, is a quadruple $(\D,\pi_\A,\pi_\B,\omega)$ where:
\begin{enumerate}
\item $\D$ is a unital C*-algebra,
\item the element $\omega$, called the \emph{pivot} of the bridge, satisfies $\omega\in\D$ and $\StateSpace_1(\D|\omega) \not=\emptyset$,
\item $\pi_\A : \A\hookrightarrow \D$ and $\pi_\B : \B\hookrightarrow\D$ are unital *-monomorphisms.
\end{enumerate}
\end{definition}

In the next few definitions, we denote by $\Haus{\mathrm{d}}$ the \emph{Hausdorff (pseudo)distance} induced by a (pseudo)distance $\mathrm{d}$ on the compact subsets of a (pseudo)metric space $(X,\mathrm{d})$ \cite{Hausdorff}.

\begin{definition}[{\cite[Definition 3.16]{Latremoliere13}}]
Let $(\A,\Lip_\A)$ and $(\B,\Lip_\B)$ be two {\gQqcms s}. The \emph{height} $\bridgeheight{\gamma}{\Lip_\A,\Lip_\B}$ of a bridge \\ $\gamma = (\D,\pi_\A,\pi_\B,\omega)$ from $\A$ to $\B$, and with respect to $\Lip_\A$ and $\Lip_\B$, is given by:
\begin{equation*}
\max\left\{ \Haus{\Kantorovich{\Lip_\A}}(\StateSpace(\A), \pi_\A^\ast(\StateSpace_1(\D|\omega))), \Haus{\Kantorovich{\Lip_\B}}(\StateSpace(\B), \pi_\B^\ast(\StateSpace_1(\D|\omega))) \right\}\text{,}
\end{equation*}
where $\pi_\A^{\ast}$ and $\pi_\B^\ast$ are the dual maps of $\pi_\A$ and $\pi_\B$, respectively.
\end{definition}

\begin{definition}[{\cite[Definition 3.10]{Latremoliere13}}]
Let $(\A,\Lip_\A)$ and $(\B,\Lip_\B)$ be two unital C*-algebras. The \emph{bridge seminorm} $\bridgenorm{\gamma}{\cdot}$ of a bridge $\gamma = (\D,\pi_\A,\pi_\B,\omega)$ from $\A$ to $\B$ is the seminorm defined on $\A\oplus\B$ by:
\begin{equation*}
\bridgenorm{\gamma}{a,b} = \|\pi_\A(a)\omega - \omega\pi_\B(b)\|_\D
\end{equation*}
for all $(a,b) \in \A\oplus\B$.
\end{definition}

We implicitly identify $\A$ with $\A\oplus\{0\}$ and $\B$ with $\{0\}\oplus\B$ in $\A\oplus\B$ in the next definition, for any two spaces $\A$ and $\B$.

\begin{definition}[{\cite[Definition 3.14]{Latremoliere13}}]
Let $(\A,\Lip_\A)$ and $(\B,\Lip_\B)$ be two {\gQqcms s}. The \emph{reach} $\bridgereach{\gamma}{\Lip_\A,\Lip_\B}$ of a bridge \\
 $\gamma = (\D,\pi_\A,\pi_\B,\omega)$ from $\A$ to $\B$, and with respect to $\Lip_\A$ and $\Lip_\B$, is given by:
\begin{equation*}
\Haus{\bridgenorm{\gamma}{\cdot}}\left( \left\{a\in\sa{\A} : \Lip_\A(a)\leq 1\right\} , \left\{ b\in\sa{\B} : \Lip_\B(b) \leq 1 \right\}  \right) \text{.}
\end{equation*}
\end{definition}

We thus choose a natural quantity to synthesize the information given by the height and the reach of a bridge:

\begin{definition}[{\cite[Definition 3.17]{Latremoliere13}}]
Let $(\A,\Lip_\A)$ and $(\B,\Lip_\B)$ be two {\gQqcms s}. The \emph{length} $\bridgelength{\gamma}{\Lip_\A,\Lip_\B}$ of a bridge $\gamma = (\D,\pi_\A,\pi_\B,\omega)$ from $\A$ to $\B$, and with respect to $\Lip_\A$ and $\Lip_\B$, is given by:
\begin{equation*}
\max\left\{\bridgeheight{\gamma}{\Lip_\A,\Lip_\B}, \bridgereach{\gamma}{\Lip_\A,\Lip_\B}\right\}\text{.}
\end{equation*}
\end{definition}

\begin{theorem-definition}[\cite{Latremoliere13, Latremoliere15}]\label{def-thm}
Fix $C\geq 1$ and $D \geq 0$. Let $\mathcal{QQCMS}_{C,D}$ be the class of all {\Qqcms{(C,D)}s}. There exists a class function $\qpropinquity{C,D}$ from $\mathcal{QQCMS}_{C,D}\times\mathcal{QQCMS}_{C,D}$ to $[0,\infty) \subseteq \R$ such that:
\begin{enumerate}
\item for any $(\A,\Lip_\A), (\B,\Lip_\B) \in \mathcal{QQCMS}_{C,D}$ we have:
\begin{equation*}
 \qpropinquity{C,D}((\A,\Lip_\A),(\B,\Lip_\B)) \leq \max\left\{\diam{\StateSpace(\A)}{\Kantorovich{\Lip_\A}}, \diam{\StateSpace(\B)}{\Kantorovich{\Lip_\B}}\right\}\text{,}
\end{equation*}
\item for any $(\A,\Lip_\A), (\B,\Lip_\B) \in \mathcal{QQCMS}_{C,D}$ we have:
\begin{equation*}
0\leq \qpropinquity{C,D}((\A,\Lip_\A),(\B,\Lip_\B)) = \qpropinquity{C,D}((\B,\Lip_\B),(\A,\Lip_\A))
\end{equation*}
\item for any $(\A,\Lip_\A), (\B,\Lip_\B), (\alg{C},\Lip_{\alg{C}}) \in \mathcal{QQCMS}_{C,D}$ we have:
\begin{equation*}
\qpropinquity{C,D}((\A,\Lip_\A),(\alg{C},\Lip_{\alg{C}})) \leq \qpropinquity{C,D}((\A,\Lip_\A),(\B,\Lip_\B)) + \qpropinquity{C,D}((\B,\Lip_\B),(\alg{C},\Lip_{\alg{C}}))\text{,}
\end{equation*}
\item for all  for any $(\A,\Lip_\A), (\B,\Lip_\B) \in \mathcal{QQCMS}_{C,D}$ and for any bridge $\gamma$ from $\A$ to $\B$, we have:
\begin{equation*}
\qpropinquity{C,D}((\A,\Lip_\A), (\B,\Lip_\B)) \leq \bridgelength{\gamma}{\Lip_\A,\Lip_\B}\text{,}
\end{equation*}
\item for any $(\A,\Lip_\A), (\B,\Lip_\B) \in \mathcal{QQCMS}_{C,D}$, we have:
\begin{equation*}
\qpropinquity{C,D}((\A,\Lip_\A),(\B,\Lip_\B)) = 0
\end{equation*}
if and only if $(\A,\Lip_\A)$ and $(\B,\Lip_\B)$ are quantum isometric, i.e. if and only if there exists a *-isomorphism $\pi : \A \rightarrow\B$ with $\Lip_\B\circ\pi = \Lip_\A$, 
\item if $\Xi$ is a class function from $\mathcal{QQCMS}_{C,D}\times \mathcal{QQCMS}_{C,D}$ to $[0,\infty)$ which satisfies Properties (2), (3) and (4) above, then:
\begin{equation*}
\Xi((\A,\Lip_\A), (\B,\Lip_\B)) \leq \qpropinquity{C,D}((\A,\Lip_\A),(\B,\Lip_\B))
\end{equation*}
 for all $(\A,\Lip_\A)$ and $(\B,\Lip_\B)$ in $\mathcal{QQCMS}_{C,D}$
\end{enumerate}
\end{theorem-definition}
 The quantum propinquity is, in fact, a special form of the dual \\ Gromov-Hausdorff propinquity \cite{Latremoliere13b, Latremoliere14, Latremoliere15} also introduced by \Latremoliere, which is a complete metric, up to isometric isomorphism, on the class of {\Lqcms s}, and which extends the topology of the Gromov-Hausdorff distance as well. Thus, as the dual propinquity is dominated by the quantum propinquity \cite{Latremoliere13b}, we conclude that \emph{all the convergence results in this paper are valid for the dual Gromov-Hausdorff propinquity as well.}

\begin{convention}
In this paper, $\qpropinquity{}$ will be meant for $\qpropinquity{C,D}$ since the Lip-norms will make the notation clear.  Also, if a Lip-norm is a $(1,0)$-Leibniz Lip-norm, then we call it a Leibniz Lip-norm.
\end{convention}

Now that the quantum Gromov-Hausdorff propinquity is defined, we provide a main result from \cite{AL}.
For our work in AF algebras, it turns out that our Lip-norms are $(2,0)$-quasi-Leibniz Lip-norms.   The following Theorem (\ref{AF-lip-norms-thm}) is \cite[Theorem 3.5]{AL}. But, first some notation.

\begin{notation}\label{ind-lim}
Let $\mathcal{I} = (\A_n,\alpha_n)_{n\in\N}$ be an inductive sequence, in which $\A_n$ is a C*-algebra and $\alpha_n$ is a *-homomorphism for all $n \in \N$, with limit $\A=\varinjlim \mathcal{I}$. We denote the canonical *-homomorphisms $\A_n \rightarrow\A$ by $\indmor{\alpha}{n}$ for all $n\in\N$, (see {\cite[Chapter 6.1]{Murphy90}}).
\end{notation}

\begin{theorem}[{\cite[Theorem 3.5]{AL}}]\label{AF-lip-norms-thm}
Let $\A$ be a unital AF algebra endowed with a faithful tracial state $\mu$. Let $\mathcal{I} = (\A_n,\alpha_n)_{n\in\N}$ be an inductive sequence of finite dimensional C*-algebras with C*-inductive limit $\A$, with $\A_0 \cong \C$ and where $\alpha_n$ is unital and injective for all $n\in\N$.

Let $\pi$ be the GNS representation of $\A$ constructed from $\mu$ on the space $L^2(\A,\mu)$.

For all $n\in\N$, let:
\begin{equation*}
\CondExp{\cdot}{\indmor{\alpha}{n}(\A_n)} : \A\rightarrow\A
\end{equation*}
be the unique conditional expectation of $\A$ onto the canonical image $\indmor{\alpha}{n}\left(\A_n\right)$ of $\A_n$ in $\A$, and such that $\mu\circ\CondExp{\cdot}{\indmor{\alpha}{n}(\A_n)} = \mu$.

Let $\beta: \N\rightarrow (0,\infty)$ have limit $0$ at infinity. If, for all $a\in\sa{\cup_{n \in \N} \indmor{\alpha}{n}\left(\A_n\right)}$, we set:
\begin{equation*}
\Lip_{\mathcal{I},\mu}^\beta(a) = \sup\left\{\frac{\left\|a - \CondExp{a}{\indmor{\alpha}{n}(\A_n)}\right\|_\A}{\beta(n)} : n \in \N \right\},
\end{equation*}
then $\left(\A,\Lip_{\mathcal{I},\mu}^\beta\right)$ is a {\Qqcms{(2,0)}}. Moreover:
\begin{enumerate}
\item
$\qpropinquity{}\left(\left(\A_n,\Lip_{\mathcal{I},\mu}^\beta\circ\indmor{\alpha}{n} \right), \left(\A,\Lip_{\mathcal{I},\mu}^\beta \right)\right) \leq \beta(n) $ for all $n \in \N$
\item and thus:
\begin{equation*}
\lim_{n\rightarrow\infty} \qpropinquity{}\left(\left(\A_n,\Lip_{\mathcal{I},\mu}^\beta\circ\indmor{\alpha}{n}\right), \left(\A,\Lip_{\mathcal{I},\mu}^\beta\right)\right) = 0\text{.}
\end{equation*}
\end{enumerate}
\end{theorem}

In \cite{AL}, the fact that the defining finite-dimensional subalgebras provide approximations of the inductive limit with respect to the quantum Gromov-Hausdorff propinquity along with Inequality (1) of Theorem (\ref{AF-lip-norms-thm}) allowed us to prove that both the UHF algebras and the Effros-Shen AF algebras are continuous images of the Baire space \cite{Miller95} with respect to the quantum propinquity.  These are \cite[Theorem 4.9]{AL} and \cite[Theorem 5.14]{AL}, respectively. 

\section{Convergence of AF algebras in the Gromov-Hausdorff propinquity}\label{conv-af}
Taking stock of our construction of Lip-norms for unital AF algebras with faithful tracial state in Theorem (\ref{AF-lip-norms-thm}), it is apparent that the construction relies on the inductive sequence, faithful tracial state, and some real-valued positive sequence converging to $0$.  Thus, this section provides suitable notions of convergence for all 3 of these structures, which in turn produce convergence of AF algebras in the quantum propinquity.  This is motivated by our arguments of continuity in \cite{AL} for the UHF and Effros-Shen AF algebras, and therefore, provides these continuity results as a consequence of Theorem (\ref{af-cont}).  We first present a basic Lemma (\ref{trace-conv-lemma}) which provides an equivalent condition for weak-* convergence of tracial states on a finite-dimensional C*-algebra.

\begin{notation}
Let $\overline{\N}=\N \cup \{\infty\}$ denote the Alexandroff compactification of $\N$ with respect to the discrete topology of $\N$. For $N \in \N$, let $\N_{\geq N} = \{ k \in \N : k \geq N \}$, and similarly, for $\overline{\N}_{\geq N}$.
\end{notation}

 \begin{notation}\label{matrix-units}
For all $d \in \N$, we denote the full matrix algebra of $d\times d$ matrices over $\C$ by $\alg{M}(d)$.

Let $\B=\oplus_{j=1}^N \alg{M}(n(j))$ for some $N \in\N$ and $n(1),\ldots,n(N) \in \N\setminus\{0\}$. For each $k\in \{1,\ldots,N\}$ and for each $j,m \in \{1,\ldots,n(k)\}$, we denote the matrix $((\delta_{u,v}^{j,m}))_{u,v  =1,\ldots, n(k)}$ by $e_{k,j,m}$, where we used the Kronecker symbol:
\begin{equation*}
\delta_a^b = \begin{cases}
1 & \text{ if $a=b$,}\\
0 & \text{ otherwise,}
\end{cases}
\end{equation*} 
where $\{e_{k,j,m} \in \B : k\in \{1,\ldots,N\},j,m \in \{1,\ldots,n(k)\}\}$ is called the set of {\em matrix units} of $\B$.
\end{notation}

\begin{lemma}\label{trace-conv-lemma}
Let $\A=\oplus_{j=1}^{N} \M(d(j))$ for some $N \in \N \setminus \{0\}$ and $d(1), \ldots, d(N) \in \N$. 

If $\{\tau^n : \A \longrightarrow \C \}_{  n \in \overline{\N}}$ is a family of tracial states, then  for each $(n,j) \in \overline{\N} \times \{1,\ldots, N \}$, there exist $\lambda_{n,j} \in [0,1]$ such that:
\begin{equation*}\tau^n (a_1, \ldots, a_N)=\sum_{j=1}^N \lambda_{n,j} \mathsf{tr}_{d(j)}(a_j), \ \forall (a_1, \ldots, a_N) \in \A,
\end{equation*} where $ \mathsf{tr}_{d(j)}$ is the unique normalized trace on $\M(d(j))$, and $(\tau^n)_{n \in \N}$ converges to $\tau^\infty$ in the weak* topology on $\StateSpace (\A)$ if and only if $((\lambda_{n,1}, \lambda_{n,2}, \ldots \lambda_{n,N}))_{n \in \N}$ converges to $(\lambda_{\infty,1}, \lambda_{\infty,2} , \ldots , \lambda_{\infty,N})$ in the product topology on $\R^N$.
\end{lemma}
\begin{proof}
The fact that tracial states of $\A$ have the above form is given by \cite[Example IV.5.4]{Davidson}.  We begin with the forward implication.  For $j \in \{1,2,\ldots, N \}$ , let $I^j=(b_1, b_2 , \ldots, b_{N}) \in \A$ such that $b_l=0$ for $l \neq j$ and $b_j= I_j$, where $I_j$ is the identity matrix in $\M(d(j))$.  By assumption, for each $j \in \{1,2,\ldots, N \}$, we have that  
\begin{equation*}
\lim_{n \to \infty} \lambda_{n,j}=\lim_{n \to \infty}\lambda_{n,j} \mathsf{tr}_{d(j)}(I_j)=  \lim_{n \to \infty}   \tau^n (I^j) = \tau^\infty (I^j) = \lambda_{\infty,j} , 
\end{equation*}  
which also provides convergence in the product topology since the product is finite. 

For the reverse implication, fix $b=(b_1, b_2 , \ldots, b_{N}) \in \A$.  Fix $n \in \N, $ then:
\begin{equation*}
\begin{split}
\left\vert \tau^n (b) -\tau^\infty(b) \right\vert &=\left\vert \left(\sum_{j=1}^{N}\lambda_{n,j}\mathsf{tr}_{d(j)}(b_j) \right) - \left(\sum_{j=1}^{N}\lambda_{\infty,j}\mathsf{tr}_{d(j)}(b_j)\right) \right\vert \\
& = \left\vert \sum_{j=1}^{N} (\lambda_{n,j}- \lambda_{\infty,j})\mathsf{tr}_{d(j)}(b_j) \right\vert \\
& \leq \left(\sum_{j=1}^{N} \vert \lambda_{n,j}- \lambda_{\infty,j} \vert\right) \Vert b \Vert_{\A}.
\end{split}
\end{equation*}
By convergence in the product topology, $\left(\sum_{j=1}^{N} \vert \lambda_{n,j}- \lambda_{\infty,j} \vert\right)_{n \in \N} $ converges to $0$.  Hence, $\lim_{n \to \infty} \left\vert \tau^n(b) - \tau^\infty (b) \right\vert = 0$ and our result is proven.
\end{proof}
Next, we consider convergence of conditional expectations on finite-dimensional C*-algebras.  We note that in the hypothesis of Proposition (\ref{cond-cont}), we now impose that our family of tracial states are faithful.
\begin{proposition}\label{cond-cont}
Let $\B$ be a unital $C^*$-algebra. Let $\A$ be a finite-dimensional unital $C^*$-subalgebra of $\B$ such that $\A \cong  \oplus_{j=1}^N \alg{M}(n(j))$ for some $N \in\N$ and $n(1),\ldots,n(N) \in \N\setminus\{0\}$ with *-isomorphism $\alpha:  \oplus_{j=1}^N \alg{M}(n(j)) \rightarrow \A$.   Let $E$ be the set of matrix units for $ \oplus_{j=1}^N \alg{M}(n(j))$ of Notation (\ref{matrix-units}).

 If $\{\tau^n: \B \rightarrow \C \}_{n \in \overline{\N}} $ is a family of faithful tracial states, then  for all $n \in \overline{\N},b \in \B$:
\begin{equation*}
\mathds{E}^n (b)=\sum_{e \in E}\frac{\tau^n(\alpha(e^*) b)}{\tau^n(\alpha(e^\ast e))}\alpha(e), 
\end{equation*}
where $\mathds{E}^n : \B \rightarrow \A$ is the unique $\tau^n$-preserving conditional expectation \cite[Definition 1.5.9]{Brown-Ozawa}.

Furthermore, if $(\tau^n  )_{n \in \N}$ converges to $\tau^\infty  $ in the weak-* topology on $\StateSpace(\B)$, then the map: 
\begin{equation*}
(n,b) \in \overline{\N} \times \B \longmapsto \Vert b-\mathds{E}^n (b) \Vert_\B  \in \R,
\end{equation*}
 is continuous with respect to the product topology on  $  \overline{\N} \times \left(\B, \Vert \cdot \Vert_\B \right)$.
\end{proposition}

\begin{proof}For $n \in \overline{\N}$, by \cite[Section 4.1]{AL}  and \cite[Expression 4.1]{AL}, we have that for each $n \in \overline{\N}$: 
\begin{equation*}
\mathds{E}^n (b)=\sum_{e \in E}\frac{\tau^n(\alpha(e)^\ast b)}{\tau^n(\alpha(e^\ast e))}\alpha(e)
\end{equation*}
since $\tau^n$ is a faithful tracial state on $\B$. 
By faithfulness, for $e \in e_\A$ and $b \in \B$,
we have $\lim_{n \to \infty}\tau^n(\alpha(e^\ast e))=\tau^\infty(\alpha(e^\ast e ))> 0$
by weak-* convergence. Since our sum is finite by finite dimensionality, again by weak-* convergence:
\begin{equation}\label{weak-tau}
\lim_{n \to \infty}\sum_{e \in E}\frac{\tau^n(\alpha(e)^\ast b)}{\tau^n(\alpha(e^\ast e))}=\sum_{e \in E}\frac{\tau^\infty(\alpha(e)^\ast b)}{\tau^\infty(\alpha(e^\ast e))}
\end{equation}
Furthermore, if we let $C=\max_{e \in E} \left\{ \Vert \alpha(e)\Vert_\B \right\}$, then:
\begin{equation*}
\begin{split}
\left\Vert \mathds{E}^n (b) - \mathds{E} (b) \right\Vert_\B =&\left\Vert \sum_{e \in E}\frac{\tau^n(\alpha(e)^\ast b)}{\tau^n(\alpha(e^\ast e))}\alpha(e)
 - \sum_{e \in E}\frac{\tau^\infty(\alpha(e)^\ast b)}{\tau^\infty(\alpha(e^\ast e))}\alpha(e)\right\Vert_\B \\
 & =   \left\Vert \sum_{e \in E}\left(\frac{\tau^n(\alpha(e)^\ast b)}{\tau^n(\alpha(e^\ast e))}- \frac{\tau^\infty(\alpha(e)^\ast b)}{\tau^\infty(\alpha(e^\ast e))}\right)\alpha(e)
\right\Vert_\B \\
& \leq     \sum_{e \in E}\left\Vert \left(\frac{\tau^n(\alpha(e)^\ast b)}{\tau^n(\alpha(e^\ast e))}- \frac{\tau^\infty(\alpha(e)^\ast b)}{\tau^\infty(\alpha(e^\ast e))}\right)\alpha(e)
\right\Vert_\B  \\
& = \sum_{e \in E}\left\vert\frac{\tau^n(\alpha(e)^\ast b)}{\tau^n(\alpha(e^\ast e))}- \frac{\tau^\infty(\alpha(e)^\ast b)}{\tau^\infty(\alpha(e^\ast e))}\right\vert \left\Vert \alpha(e)
\right\Vert_\B \\
&\leq \left(\sum_{e \in E} \left\vert  \frac{\tau^n(\alpha(e)^\ast b)}{\tau^n(\alpha(e^\ast e))}-\frac{\tau^\infty(\alpha(e)^\ast b)}{\tau^\infty(\alpha(e^\ast e))}\right\vert \right)C, 
\end{split}
\end{equation*}
and  $\lim_{n \to \infty} \Vert \mathds{E}^n(b) - \mathds{E}(b) \Vert_\B =0$  by Expression (\ref{weak-tau}).

Fix $n,m \in \overline{\N}$ and $b,b' \in \B$.  Then, as conditional expectations are contractive \cite[1.5.10]{Brown-Ozawa}: 
\begin{equation*}
\begin{split}
\left\vert \left\Vert b-\mathds{E}^n(b) \right\Vert_\B - \left\Vert b'-\mathds{E}^m (b') \right\Vert_\B \right\vert & \leq \left\Vert (b-\mathds{E}^n(b)) - (b'- \mathds{E}^m (b')) \right\Vert_\B \\
& \leq  \left\Vert \mathds{E}^n(b) - \mathds{E}^n(b') +\mathds{E}^n(b')- \mathds{E}^m (b') \right\Vert_\B \\
& \quad +  \left\Vert b-b' \right\Vert_\B \\
& \leq 2 \left\Vert b-b' \right\Vert_\B +\left\Vert \mathds{E}^n(b')- \mathds{E}^m (b') \right\Vert_\B , 
\end{split}
\end{equation*}
and continuity follows.
\end{proof}
We now introduce an appropriate notion of merging  inductive sequences together in Definition (\ref{coll-def}).
\begin{definition}\label{coll-def} We consider 2 cases of inductive sequences in this definition.
\begin{case}Closure of union
\end{case}
For each $k \in \overline{\N}$, let $\A^k$ be a C*-algebras with  $\A^k = \overline{\cup_{n \in \N} \A_{k,n} }^{\Vert \cdot \Vert_{\A^k}}$ such that $\mathcal{U}^k = \left(\A_{k,n}\right)_{n \in \N}$ is a non-decreasing sequence of  C*-subalgebras of $\A^k$, then we say $\{\A^k : k \in \overline{\N}\}$ is a {\em fusing family} if: 
\begin{enumerate} 
\item There exists  $(c_n)_{n \in \N} \subseteq \N$ non-decreasing such that $\lim_{n \to \infty} c_n=\infty$, and
\item for all $N \in \N $, if $k \in \N_{\geq c_N}$, then $\A_{k,n} = \A_{\infty,n}$ for all $n \in \{0,1,\ldots, N\}.$
\end{enumerate}

\begin{case}Inductive limit
\end{case}
For each $k \in \overline{\N}$, let  $\mathcal{I}(k)=(\A_{k,n}, \alpha_{k,n})_{n \in \N}$ be an inductive sequence with inductive limit, $\A^k$.  We say that the family of $C^*$-algebras $\{ \A^k : k \in \overline{\N} \}$ is an {\em IL-fusing family} of $C^*$-algebras if:
\begin{enumerate} 
\item There exists  $(c_n)_{n \in \N} \subseteq \N$ non-decreasing such that $\lim_{n \to \infty} c_n=\infty$, and
\item for all $N \in \N $, if $k \in \N_{\geq c_N}$, then $(\A_{k,n}, \alpha_{k,n}) = (\A_{\infty,n}, \alpha_{\infty,n})$ for all $n \in \{0,1,\ldots, N\}.$
\end{enumerate}

In either case, we call the sequence  $(c_n)_{n \in \N} $ the {\em fusing sequence}.
\end{definition}

\begin{remark}
The results in this section are phrased in terms of {\em IL-fusing families} since our propinquity convergence results are all in terms of inductive limits.  But, we note that all the results of this section are valid for the closure of union case as well with appropriate translations, but most convergence results are convergence results are more easily fulfilled in the inductive limit case.   Note that any IL-fusing family may be viewed as a fusing family via the canonical *-homomorphisms of Notation (\ref{ind-lim}), which is why we don't decorate the term {\em fusing family} in the closure of union case. 
\end{remark}

Hypotheses (2) and (3) in the following Lemma (\ref{lip-cont}) introduce the remaining notions of convergence that together with fusing families will imply convergence of quantum propinquity of AF algebras in Theorem (\ref{af-cont}).  Indeed, (2) is simply an appropriate use of weak-* convergence for the faithful tracial states in relation to fusing families, and (3)  is an appropriate use of pointwise convergence of the sequences that provide convergence of the finite dimensional subspaces in Theorem (\ref{AF-lip-norms-thm}).  

Furthermore, Lemma (\ref{lip-cont}) provides that the Lip-norms induced on the finite dimensional subspaces form a continuous field of Lip-norms, a notion introduced by Rieffel in \cite{Rieffel00}.
\begin{lemma}\label{lip-cont} For each $k \in \overline{\N}$, let  $\mathcal{I}(k)=(\A_{k,n}, \alpha_{k,n})_{n \in \N}$ be an inductive sequence of finite dimensional $C^*$-algebras  with $C^*$-inductive limit $\A^k ,$ such that $\A_{k,0}=\A_{k',0} \cong \C$ and $\alpha_{k,n}$ is  unital and injective for all $k ,k' \in \overline{\N}, n\in\N$.  If:
\begin{enumerate}
\item $\{ \A^k : k \in \overline{\N} \}$ is an IL-fusing family with fusing sequence $(c_n)_{n \in \N}$,
\item $\{\tau^k : \A^k \rightarrow \C \}_{k \in \overline{\N}}$ is a family of faithful tracial states such that for each $N \in \N$, we have that $\left(\tau^k \circ \indmor{\alpha_k}{N}\right)_{k \in \N_{\geq c_N}}$ converges to $\tau^\infty \circ \indmor{\alpha_\infty}{N}$ in the weak-* topology on $\StateSpace(\A_{\infty,N})$, and
\item $\{\beta^k : \overline{\N} \rightarrow (0,\infty) \}_{k \in \overline{\N}}$ is a family of convergent sequences such that for all $N\in \N $ if $k \in \N_{\geq c_N}$, then  $\beta^k(n)=\beta^\infty(n)$ for all $n \in \{0,1,\ldots,N\}$ and there exists $B: \overline{\N} \rightarrow (0,\infty)$ with $B(\infty)=0$ and $\beta^m(l) \leq B(l)$ for all $m,l \in \overline{\N}$,
\end{enumerate}
then for all $N \in \N$, if $n \in \{0,1, \ldots, N\}$, then the map:
\begin{equation*}\mathsf{l}^N_n : (k,a) \in \overline{\N}_{\geq c_N} \times \A_{\infty,n} \longmapsto \Lip^{\beta^k}_{\mathcal{I}(k),\tau^k}\circ \indmor{\alpha_k}{n} (a) \in \R
\end{equation*}
is well-defined and continuous with respect to the product topology on \\ $\overline{\N} \times \left(\A_{\infty,n},\Vert \cdot \Vert_{ \A_{\infty,n}}\right)$, where $\Lip^{\beta^k}_{\mathcal{I}(k),\tau^k}$ is given by Theorem (\ref{AF-lip-norms-thm}).
\end{lemma}
\begin{proof} First, we establish a weak-* convergence result implied by  (2). Let $N \in \overline{\N}$.
\begin{claim}\label{lower-space} $\left(\tau^k \circ \indmor{\alpha_k}{m}\right)_{k \in \N_{\geq c_N}}$ converges to $\tau^\infty \circ \indmor{\alpha_\infty}{m}$ in the weak* topology on $\StateSpace(\A_{\infty,m})$ for each $m \in \{0,1,\ldots, N\}$. 
\end{claim} \begin{proof}[Proof of claim]
Let $m \in \{0,1,\ldots, N\}$. The case $m=N$ is given by assumption. So, assume that $N \geq 1$ and $m \in \{0, \ldots , N-1\}$. Fix $a \in \A_{\infty,m}$,  we have by definition of inductive limit and IL-fusing family:
\begin{equation*} \tau^k \circ \indmor{\alpha_k}{m}(a)=\tau^k \circ \indmor{\alpha_k}{N}(\alpha_{k,N-1} \circ \cdots \circ \alpha_{k,m} (a))=\tau^k \circ \indmor{\alpha_k}{N}(\alpha_{\infty,N-1} \circ \cdots \circ \alpha_{\infty,m} (a))
\end{equation*} for  $k \in \overline{\N}_{\geq c_N}$,  which proves our claim since $\left(\tau^k \circ \indmor{\alpha_k}{N}\right)_{k \in \N_{\geq c_N}}$ converges to $\tau^\infty \circ \indmor{\alpha_\infty}{N}$ in the weak* topology on $\StateSpace(\A_{\infty,N})$. \end{proof}

Next, we establish a more explicit form of our Lip-norms on the finite-dimensional subspaces. Fix $N \in \N$ and $n \in  \{0,1, \ldots, N\}$.  $\mathsf{l}^N_n$ is well-defined by definition of a IL-fusing family.  Furthermore, as $ \CondExp{\cdot}{\indmor{\alpha_{k}}{j}(\A_{k,j})}$ is a conditional expectation for all $k \in \overline{\N}, j \in \N$, we have that: 
\begin{equation*}\CondExp{\indmor{\alpha_{k}}{n}(a)}{\indmor{\alpha_{k}}{j}(\A_{k,j})}=\indmor{\alpha_{k}}{n}(a)
\end{equation*} for $j \geq n, a \in \A_{\infty,n}$ by \cite[Definition 1.5.9]{Brown-Ozawa}.

Therefore:
\begin{equation}\label{fd-lip}
\mathsf{l}^N_n(k,a)=\max \left\{ \frac{\left\Vert \indmor{\alpha_{k}}{n}(a) -   \CondExp{\indmor{\alpha_{k}}{n}(a)}{\indmor{\alpha_{k}}{m}(\A_{\infty,m})}\right\Vert_{\A^k}}{\beta^\infty(m)} : m \in \{0,\ldots,n-1\}\right\}.
\end{equation}

Fix $m \in \{0,\ldots, n-1\}, k\geq N, a \in \A_{\infty,n}$.  Since $\A_{\infty,m}$ is finite dimensional, the C*-algebra $\A_{\infty,m} \cong  \oplus_{j=1}^N \alg{M}(n(j))$ for some $N \in\N$ and $n(1),\ldots,n(N) \in \N\setminus\{0\}$ with *-isomorphism $\gamma:  \oplus_{j=1}^N \alg{M}(n(j)) \rightarrow \A_{\infty,m}$.   Let $E$ be the set of matrix units for $ \oplus_{j=1}^N \alg{M}(n(j))$. Now, define $\alpha_{k,m\to n-1}= \alpha_{k,n-1} \circ \cdots \circ \alpha_{k,m}$, and by definition of IL-fusing family, we have that $\alpha_{k,m\to n-1} =\alpha_{\infty,m \to n-1}$. Therefore, by the definition of inductive limit and  Proposition (\ref{cond-cont}):  
\begin{equation}\label{lip-limit}
\begin{split}
& \left\Vert \indmor{\alpha_{k}}{n}(a) -   \CondExp{\indmor{\alpha_{k}}{n}(a)}{\indmor{\alpha_{k}}{m}(\A_{\infty,m})}\right\Vert_{\A^k} \\
& = \Bigg\Vert \indmor{\alpha_{k}}{n}(a) -\sum_{e \in E} \frac{\tau^k(\indmor{\alpha_{k}}{m} \circ \gamma(e^*) \indmor{\alpha_{k}}{n}(a))}{\tau^k(\indmor{\alpha_{k}}{m} \circ \gamma(e^\ast e))}\indmor{\alpha_{k}}{m} \circ \gamma(e) \Bigg\Vert_{\A^k}\\
& = \Bigg\Vert \indmor{\alpha_{k}}{n}(a)  - \sum_{e \in E} \frac{\tau^k(\indmor{\alpha_{k}}{n}(\alpha_{k,m\to n-1}(\gamma(e)^\ast) a))}{\tau^k(\indmor{\alpha_{k}}{n}(\alpha_{k,m\to n-1}(\gamma(e^\ast e)))}\indmor{\alpha_{k}}{n}(\alpha_{k,m\to n-1}(\gamma(e))) \Bigg\Vert_{\A^k}\\
& =\Bigg\Vert \indmor{\alpha_{k}}{n}\Bigg(a - \sum_{e \in E} \frac{\tau^k(\indmor{\alpha_{k}}{n}(\alpha_{\infty,m\to n-1}(\gamma(e)^\ast) a))}{\tau^k(\indmor{\alpha_{k}}{n}(\alpha_{\infty,m\to n-1}(\gamma(e^*e)))}\alpha_{\infty,m\to n-1}(\gamma(e))\Bigg) \Bigg\Vert_{\A^k}\\
& = \Bigg\Vert a - \sum_{e \in E} \frac{\tau^k\circ \indmor{\alpha_{k}}{n}(\alpha_{\infty,m\to n-1}\circ \gamma(e^*) a)}{\tau^k\circ \indmor{\alpha_{k}}{n}(\alpha_{\infty,m\to n-1}\circ \gamma(e^\ast e))}\alpha_{\infty,m\to n-1}\circ \gamma(e) \Bigg\Vert_{\A_{\infty,n}}
\end{split}
\end{equation} 
Hence, by Claim (\ref{lower-space}) and Proposition (\ref{cond-cont}), the map: 
\begin{equation*}
(k,a) \in  \overline{\N}_{\geq c_N} \times \A_{\infty,n} \mapsto \frac{\left\Vert \indmor{\alpha_{k}}{n}(a) -   \CondExp{\indmor{\alpha_{k}}{n}(a)}{\indmor{\alpha_{k}}{m}(\A_{\infty,m})}\right\Vert_{\A^k}}{\beta^\infty(m)} \in \R
\end{equation*}
is continuous for each $m \in \{0, \ldots, n-1\}$.  As the maximum of finitely many continuous real-valued functions is continuous, our lemma is proven by Expression (\ref{fd-lip}). 
\end{proof}

The proof of the following Theorem (\ref{fd-cont}) follows similarly as the proof of {\cite[Lemma 5.13]{AL}}, but we include the proof for convenience and clarity.  This Theorem (\ref{fd-cont}) establishes convergence of the finite dimensional subspaces.    

\begin{theorem}\label{fd-cont}For each $k \in \overline{\N}$, let  $\mathcal{I}(k)=(\A_{k,n}, \alpha_{k,n})_{n \in \N}$ be an inductive sequence of finite dimensional $C^*$-algebras  with $C^*$-inductive limit $\A^k ,$ such that $\A_{k,0} =\A_{k',0} \cong \C$ and $\alpha_{k,n}$ is unital and injective for all $k ,k' \in \overline{\N}, n\in\N$.  If:
\begin{enumerate}
\item $\{ \A^k : k \in \overline{\N} \}$ is an IL-fusing family with fusing sequence $(c_n)_{n \in \N}$,
\item $\{\tau^k : \A^k \rightarrow \C \}_{k \in \overline{\N}}$ is a family of faithful tracial states such that for each $N \in \N$, we have that $\left(\tau^k \circ \indmor{\alpha_k}{N}\right)_{k \in \N_{\geq c_N}}$ converges to $\tau^\infty \circ \indmor{\alpha_\infty}{N}$ in the weak-* topology on $\StateSpace(\A_{\infty,N})$, and
\item $\{\beta^k : \overline{\N} \rightarrow (0,\infty) \}_{k \in \overline{\N}}$ is a family of convergent sequences such that for all $N\in \N $ if $k \in \N_{\geq c_N}$, then  $\beta^k(n)=\beta^\infty(n)$ for all $n \in \{0,1,\ldots,N\}$ and there exists $B: \overline{\N} \rightarrow (0,\infty)$ with $B(\infty)=0$ and $\beta^m(l) \leq B(l)$ for all $m,l \in \overline{\N}$,
\end{enumerate}
then for every $N \in \N$ and $n \in \{0,\ldots , N\}$, the map: \begin{equation*}
F^{N}_{n}: k \in \overline{\N}_{\geq c_N} \longmapsto \left(\A_{k, n} ,\Lip^{\beta^k}_{\mathcal{I}(k), \tau^k} \circ \indmor{\alpha_k}{n}\right) \in (\mathcal{QQCMS}_{2,0}, \qpropinquity{}) 
\end{equation*}
is well-defined and continuous, and therefore:
\begin{equation*}\lim_{k \to \infty} \qpropinquity{} \left(\left(\A_{k, n} ,\Lip^{\beta^k}_{\mathcal{I}(k), \tau^k} \circ \indmor{\alpha_k}{n}\right),\left(\A_{\infty, n} ,\Lip^{\beta^\infty}_{\mathcal{I}(\infty), \tau^\infty} \circ \indmor{\alpha_\infty}{n}\right)\right)=0,
\end{equation*}
where $\Lip^{\beta^k}_{\mathcal{I}(k), \tau^k}$ is given by Theorem (\ref{AF-lip-norms-thm}).
\end{theorem}

\begin{proof}
Fix $N \in \N$ and $n \in \{0, \ldots , N\}$.  If $n=0$, then $\A_{k,0}=\A_{\infty,0} \cong \C$ and since Lip-norms vanish only on scalars by Definition (\ref{quasi-Monge-Kantorovich-def}), the map $F^{N}_{0}$ is constant up to quantum isometry and therefore continuous.  

Assume that $n \in \{1, \ldots, N \}$ and $k \geq c_N$.  Set $\B_n = \A_{k,n}=\A_{\infty,n}$ by definition of IL-fusing family.  Let $\alg{W}$ be any complementary subspace of $\R\unit_\A$ in $\sa{\B_n}$ --- which exists since $\sa{\B_n}$ is finite dimensional. We shall denote by $\alg{S}$ the unit sphere $\{ a\in \alg{W} : \|a\|_{\B_n} = 1 \}$ in $\alg{W}$. Note that since $\alg{W}$ is finite dimensional, $\alg{S}$ is a compact set.  Set $S=\overline{\N}_{\geq c_N} \times \alg{S}$, which is compact in the product topology.  Therefore, since the function $\mathsf{I}^N_n $ is continuous by Lemma (\ref{lip-cont}), it reaches a minimum on $S$. Thus, there exists $(K,c) \in S$ such that: $\min_{s \in S}\mathsf{l}^N_n(s) = \mathsf{l}^N_n(K,c)$. In particular, since Lip-norms are zero only on the scalars, we have $\mathsf{l}^N_n(K,c) > 0$ as $\|c\|_{\alg{W}} = 1$ yet the only scalar multiple of $\unit_{\B_n}$ in $\alg{W}$ is $0$. We denote $m_S = \mathsf{l}^N_n(K,c) > 0$ in the rest of this proof.

Moreover, the function $\mathsf{l}^N_n$ is continuous on the compact set $S$, and thus, it is uniformly continuous with respect to any metric that metrizes the product topology.  In particular, consider the max metric, denoted by $\mathsf{m}$, with respect to the norm on $\alg{S}$ and the metric on $\overline{\N}$ defined by $\mathsf{d_A}(n,m)=\vert 1/(n+1) - 1/(m+1) \vert$ for all $n,m \in \overline{\N}$ with the convention that $1/(\infty +1)=0$, where the metric $\mathsf{d_A}$ metrizes the topology on $\overline{\N}$.

Let $\varepsilon > 0$. As $\mathsf{l}^N_n$ is uniformly continuous on the metric space $(S, \mathsf{m})$, there exists $\delta>0 $ such that if $\mathsf{m}(s,s')<\delta$, then $|\mathsf{l}^N_n(s) - \mathsf{l}^N_n(s')| \leq m_S^2 \varepsilon$. Now, there exsits $M \in \N_{\geq c_N}$ such that $1/M < \delta$.  Let $m \geq M$ and $a \in \alg{S}$, then by definition of the metrics $\mathsf{m}$ and $\mathsf{d_A}$: \begin{equation*}
\mathsf{m}((m,a),(\infty,a))=1/(m+1) < 1/m \leq 1/M< \delta.
\end{equation*}  Thus,  for all $m\geq M$ and for all $a \in \alg{S}$ we have:
\begin{equation*}
|\mathsf{l}^N_n(m, a) - \mathsf{l}^N_n(\infty,a)| \leq m_S^2 \varepsilon\text{.}
\end{equation*}

We then have, for all $a\in\alg{S}$ and $m\geq M$, since $\mathsf{l}^N_n$ is positive on $S$:
\begin{equation}\label{m-inf}
\begin{split}\left\| a - \frac{\mathsf{l}^N_n(m,a)}{\mathsf{l}^N_n(\infty,a)} a \right\|_{\B_n} &= \frac{|\mathsf{l}^N_n(m,a)-\mathsf{l}^N_n(\infty,a)|}{\mathsf{l}^N_n(\infty,a)} \|a\|_{\B_n} \\
&\leq \frac{\varepsilon m_S^2}{m_S} \leq m_S \varepsilon \text{.}
\end{split}
\end{equation}
Similarly:
\begin{equation}\label{af-theta-fd-lemma-eq2}
\left\| a - \frac{\mathsf{l}^N_n(\infty,a)}{\mathsf{l}^N_n(m,a)} a \right\|_{\B_n} \leq m_S \varepsilon \text{.}
\end{equation}

We are now ready to provide an estimate for the quantum propinquity. Let $m\geq M$ be fixed. Writing $\mathrm{id}$ for the identity of $\B_n$, the quadruple:
\begin{equation*}
\gamma = \left(\B_n, \unit_{\B_n}, \mathrm{id}, \mathrm{id}\right)
\end{equation*}
is a bridge in the sense of Definition (\ref{bridge-def}) from $\left(\B_n,\Lip^{\beta^m}_{\mathcal{I}(m),\tau^m}\circ \indmor{\alpha_m}{n}\right)$ to \\ $\left(\B_n,\Lip^{\beta^\infty}_{\mathcal{I}(\infty),\tau^\infty}\circ \indmor{\alpha_\infty}{n}\right)$.

As the pivot of $\gamma$ is the unit, the height of $\gamma$ is null. We are left to compute the reach of $\gamma$.

Let $a\in \sa{\B_n}$. We proceed with three case.

\setcounter{case}{0}
\begin{case}
Assume that $a \in \R\unit_{\B_n}$. 
\end{case}
We then have that $\mathsf{l}^N_n(\infty,a) =\mathsf{l}^N_n(m,a) = 0$ , and that $\|a - a\|_{\B_n} = 0$.

\begin{case}
Assume that $a \in\alg{S}$.
\end{case}
We note again that $\mathsf{l}^N_n(\infty,a) \geq m_S > 0$.  Thus, we may define $a'=\frac{\mathsf{l}^N_n(\infty,a)}{\mathsf{l}^N_n(m,a)} a$.  By Inequality (\ref{af-theta-fd-lemma-eq2}), we have:
\begin{equation*}
\left\Vert a-a'\right\Vert_{\B_n} = \left\|a - \frac{\mathsf{l}^N_n(\infty,a)}{\mathsf{l}^N_n(m,a)} a\right\|_{\B_n}\leq \varepsilon m_S \leq \varepsilon \mathsf{l}^N_n(\infty,a) \text{,}
\end{equation*}
while $\mathsf{l}^N_n\left(m,a'\right)=\mathsf{l}^N_n\left(m, \frac{\mathsf{l}^N_n(\infty,a)}{\mathsf{l}^N_n(m,a)} a \right) = \mathsf{l}^N_n(\infty,a)$.
 
\begin{case}
Assume that $a\in \sa{\B_n}$.
\end{case}

By definition of $\alg{S}$ there exists $r, t \in \R$ such that $a = rb + t\unit_{\B_n}$ with $b\in \alg{S}$.  We may assume $r\neq 0$ since the case $r=0$ would be Case 1. If $r <0$, then $-r>0, -b \in \alg{S}$ and $a=-r(-b) +t 1_{\B_n}$.  Hence, we may assume that $r>0$.  

Since Lip-norms vanish on scalars, note that $\mathsf{l}_n^N(\infty,a)= \mathsf{l}_n^N (\infty, rb)$.  Let $b' \in\sa{\B_n}$ be constructed from $b \in \alg{S}$ as in Case 2.   Now, consider $a'=rb'+t1_{\B_n}$.  Thus, by Case 2 and $r>0$: 
\begin{equation*}
\begin{split}
\left\Vert a-a' \right\Vert_{\B_n} & = \left\Vert  rb + t\unit_{\B_n}- \left(rb'+t1_{\B_n}\right) \right\Vert_{\B_n} \\
& =r \left\Vert b-b' \right\Vert_{\B_n}\\
& \leq r\varepsilon \mathsf{l}_n^N (\infty, b) \\
& = \varepsilon \mathsf{l}_n^N (\infty, rb)=\varepsilon \mathsf{l}_n^N (\infty,a), 
\end{split}
\end{equation*}
while $\mathsf{l}_n^N(m,a')=\mathsf{l}_n^N (m,rb') =r \mathsf{l}_n^N (m,b') \leq r \mathsf{l}_n^N (\infty,b)=\mathsf{l}_n^N (\infty,rb)= \mathsf{l}_n^N(\infty,a)$ by Case 2, $r>0$, and since Lip-norms vanish on scalars.

 Thus, from Case 3, we conclude that:
\begin{equation}\label{af-theta-fd-lemma-eq3}
\forall a\in \sa{\B_n}, \ \exists a' \in\sa{\B_n}  \text{ with }  \|a-a'\|_{\B_n}\leq\varepsilon \mathsf{l}^N_n(\infty,a),\ \mathsf{l}^N_n(m,a')\leq\mathsf{l}^N_n(\infty,a)\text{.}
\end{equation}

By symmetry in the roles of $\infty$ and $m$ and Inequality (\ref{m-inf}), we can conclude as well that:
\begin{equation}\label{af-theta-fd-lemma-eq4}
\forall a\in \sa{\B_n}, \ \exists a' \in\sa{\B_n} \text{with}  \|a-a'\|_{\B_n}\leq\varepsilon \mathsf{l}^N_n(m,a),\  \mathsf{l}^N_n(\infty,a')\leq\mathsf{l}^N_n(m,a)\text{.} 
\end{equation}

Now, Expressions (\ref{af-theta-fd-lemma-eq3}) and (\ref{af-theta-fd-lemma-eq4}) together imply that the reach, and hence the length of the bridge $\gamma$ is no more than $\varepsilon$.

Thus, by definition of length and Theorem-Definition (\ref{def-thm}), we gather:
\begin{equation*}
\qpropinquity{}\left((\B_n,\mathsf{l}^N_n(\infty,\cdot)), (\B_n, \mathsf{l}^N_n(m,\cdot))\right) \leq \varepsilon
\end{equation*}
for all $m\geq M,$
which concludes our proof.
\end{proof}

Next, we are now in a position to provide criteria for convergence of AF algebras in quantum propinquity.

\begin{theorem}\label{af-cont}For each $k \in \overline{\N}$, let  $\mathcal{I}(k)=(\A_{k,n}, \alpha_{k,n})_{n \in \N}$ be an inductive sequence of finite dimensional $C^*$-algebras  with $C^*$-inductive limit $\A^k ,$ such that $\A_{k,0} =\A_{k',0} \cong \C$ and $\alpha_{k,n}$ is unital and injective for all $k ,k' \in \overline{\N}, n\in\N$.  If:
\begin{enumerate}
\item $\{ \A^k : k \in \overline{\N} \}$ is an IL-fusing family with fusing sequence $(c_n)_{n \in \N}$,
\item $\{\tau^k : \A^k \rightarrow \C \}_{k \in \overline{\N}}$ is a family of faithful tracial states such that for each $N \in \N$, we have that $\left(\tau^k \circ \indmor{\alpha_k}{N}\right)_{k \in \N_{\geq c_N}}$ converges to $\tau^\infty \circ \indmor{\alpha_\infty}{N}$ in the weak-* topology on $\StateSpace(\A_{\infty,N})$, and
\item $\{\beta^k : \overline{\N} \rightarrow (0,\infty) \}_{k \in \overline{\N}}$ is a family of convergent sequences such that for all $N\in \N $ if $k \in \N_{\geq c_N}$, then  $\beta^k(n)=\beta^\infty(n)$ for all $n \in \{0,1,\ldots,N\}$ and there exists $B: \overline{\N} \rightarrow (0,\infty)$ with $B(\infty)=0$ and $\beta^m(l) \leq B(l)$ for all $m,l \in \overline{\N}$
\end{enumerate}
then, for each $N \in \N$, we have for all $k \geq c_N$:
\begin{equation}\label{fd-ineq}
 \qpropinquity{} \left(\left(\A^k , \Lip^{\beta^k}_{\mathcal{I}(k),\tau^k}\right),\left(\A^\infty , \Lip^{\beta^\infty}_{\mathcal{I}(\infty),\tau^\infty}\right)\right) \leq 2 B(N)+\qpropinquity{} \left(F^N_{N}(k) , F^N_{N}(\infty)  \right), 
\end{equation}
where$\Lip^{\beta^k}_{\mathcal{I}(k), \tau^k}$ is given by Theorem (\ref{AF-lip-norms-thm}) and  $F^N_N(k)$ is given by Theorem (\ref{fd-cont}).

Furthermore:
\begin{equation*}
\lim_{k \to \infty} \qpropinquity{} \left(\left(\A^k , \Lip^{\beta^k}_{\mathcal{I}(k),\tau^k}\right),\left(\A^\infty , \Lip^{\beta^\infty}_{\mathcal{I}(\infty),\tau^\infty}\right)\right)=0.
\end{equation*}
\end{theorem}
\begin{proof}
Fix $N \in \N $.  Then, for all $k \in \overline{\N}$: 
\begin{equation*}
\qpropinquity{} \left(\left(\A^k , \Lip^{\beta^k}_{\mathcal{I}(k),\tau^k}\right),\left(\A_{k,N} , \Lip^{\beta^k}_{\mathcal{I}(k),\tau^k} \circ \indmor{\alpha_k}{N}\right)\right) \leq \beta^k(N) \leq B(N)
\end{equation*}
by assumption and Theorem (\ref{AF-lip-norms-thm}).
And, by the triangle inequality:
\begin{equation*}
\begin{split}
 & \quad \qpropinquity{} \left(\left(\A^k , \Lip^{\beta^k}_{\mathcal{I}(k),\tau^k}\right),\left(\A^\infty , \Lip^{\beta^\infty}_{\mathcal{I}(\infty),\tau^\infty}\right)\right) \\ 
 & \leq 2B(N)
  + \qpropinquity{} \left(\left(\A_{k, N} ,\Lip^{\beta^k}_{\mathcal{I}(k), \tau^k} \circ \indmor{\alpha_k}{N}\right),\left(\A_{\infty, N} ,\Lip^{\beta^\infty}_{\mathcal{I}(\infty), \tau^\infty} \circ \indmor{\alpha_\infty}{N}\right)\right)
\end{split}
\end{equation*}
Now, assume $ k \geq c_N$.  Then, we have: 
\begin{equation*}
\qpropinquity{} \left(\left(\A^k , \Lip^{\beta^k}_{\mathcal{I}(k),\tau^k}\right),\left(\A^\infty , \Lip^{\beta^\infty}_{\mathcal{I}(\infty),\tau^\infty}\right)\right)\leq 2B(N) + \qpropinquity{} \left(F^N_{N}(k) , F^N_{N}(\infty)  \right),
\end{equation*}
and:
\begin{equation*}
\limsup_{\substack{k \to \infty \\ k \in \N_{\geq c_N}}} \qpropinquity{} \left(\left(\A^k , \Lip^{\beta^k}_{\mathcal{I}(k),\tau^k}\right),\left(\A^\infty , \Lip^{\beta^\infty}_{\mathcal{I}(\infty),\tau^\infty}\right)\right) \leq 2B(N), 
\end{equation*}
since $F^N_{N}$ is continuous by Theorem (\ref{fd-cont}). And, thus:
\begin{equation}\label{limsup}
\limsup_{k \to \infty} \qpropinquity{} \left(\left(\A^k , \Lip^{\beta^k}_{\mathcal{I}(k),\tau^k}\right),\left(\A^\infty , \Lip^{\beta^\infty}_{\mathcal{I}(\infty),\tau^\infty}\right)\right) \leq 2B(N).
\end{equation}
Hence, as the left hand side of Inequality (\ref{limsup}) does not depend on $N$, we gather: 
\begin{equation*}
 \quad \limsup_{k \to \infty} \qpropinquity{} \left(\left(\A^k , \Lip^{\beta^k}_{\mathcal{I}(k),\tau^k}\right),\left(\A^\infty , \Lip^{\beta^\infty}_{\mathcal{I}(\infty),\tau^\infty}\right)\right) 
 \leq \lim_{N \to \infty} 2B(N) = 0,
\end{equation*}
which concludes the proof.
\end{proof}

Theorem (\ref{af-cont}) provides a satisfying insight to the quantum metric structure of the Lip-norms of Theorem (\ref{AF-lip-norms-thm}).  Indeed, hypotheses (1), (2), and (3) of Theorem (\ref{af-cont}) are simply appropriate notions of convergence relying on the criteria used to construct the  Lip-norms of Theorem (\ref{AF-lip-norms-thm}) and nothing more.  

Another powerful and immediate consequence of Theorem (\ref{af-cont}) is that, in the Effros-Shen AF algebra case, since the proof of \cite[Theorem 5.14]{AL}  uses sequential continuity and convergence of irrationals in the Baire Space , it is not difficult to see how one may use this Theorem (\ref{af-cont}) to achieve the same result and we present a proof of this here in Theorem (\ref{af-theta-thm-new}).  For the, UHF case \cite[Theorem 4.9]{AL}, one could also apply  Theorem (\ref{af-cont}) to achieve continuity, but although Theorem (\ref{af-cont}) does not directly provide the fact that the map in \cite[Theorem 4.9]{AL} is H{\"o}lder, one may use  Inequality (\ref{fd-ineq}) in the statement of Theorem (\ref{af-cont}), to deduce such a result.   

We present the Effros-Shen AF Algebra case here  as Theorem (\ref{af-theta-thm-new}) to display how one may use the results of this section to prove \cite[Theorem 5.14]{AL} with ease. We note that another application of Theorem (\ref{af-cont}) is  used in \cite{Aguilar16} to provide convergence of quotients via convergence of ideals in a suitable setting. 

Although the following proof of Theorem (\ref{af-theta-thm-new}) cites results from our previous paper \cite{AL}, the results used from \cite{AL} only pertain to the metric structure of the Baire space and the definition of the faithful tracial states on the finite dimensional subalgebras and are not the convergence results themselves.  For convenience, let's first recall the definition of the Baire space and the definition of the Effros-Shen AF algebras.  
\begin{definition}[\cite{Miller95}]\label{Baire-Space-def}
The \emph{Baire space} $\BaireSpace$ is the set $(\N\setminus\{0\})^\N$ endowed with the metric $\mathsf{d}_\BaireSpace$ defined, for any two $(x(n))_{n\in\N}$, $(y(n))_{n\in\N}$ in $\BaireSpace$, by:
\begin{equation*}
\mathsf{d}_\BaireSpace \left((x(n))_{n\in\N}, (y(n))_{n\in\N}\right) = \begin{cases}
0 & \text{ if $\forall n \in \N, x(n) = y(n)$},\\
2^{-\min\left\{ n \in \N : x(n) \not= y(n) \right\}} &  \text{ otherwise}\text{.}
\end{cases}
\end{equation*}
\end{definition}

We now present the construction of the AF C*-algebras $\alg{AF}_\theta$ constructed in \cite{Effros80b} for any irrational $\theta$ in $(0,1)$. For any $\theta \in (0,1)\setminus\Q$, let $(a_j)_{j\in\N}$ be the unique sequence in $\N$ such that:
\begin{equation}\label{continued-fraction-eq}
\theta = \lim_{n\rightarrow\infty} a_0 + \cfrac{1}{a_1 + \cfrac{1}{a_2 + \cfrac{1}{a_3 +\cfrac{1}{\ddots+\cfrac{1}{a_n}}}}}=\lim_{n \to \infty} [a_0, a_1, \ldots,a_n]\text{.}
\end{equation}
The sequence $(a_j)_{j\in\N}$ is called the continued fraction expansion of $\theta$, and we will simply denote it by writing $\theta = [a_0 , a_1 , a_2, \ldots ] = [a_j]_{j\in\N}$. We note that $a_0 = 0$ (since $\theta\in(0,1)$) and $a_n \in \N\setminus\{0\}$ for $n \geq 1 $.

We fix $\theta \in (0,1)\setminus\Q$, and let $\theta = [a_j]_{j\in\N}$ be its continued fraction decomposition. We then obtain a sequence $\left(\frac{p_n^\theta}{q_n^\theta}\right)_{n\in\N}$ with $p_n^\theta \in \N$ and $q_n^\theta \in \N\setminus\{0\}$ by setting:
\begin{equation}\label{pq-rel-eq}
\begin{cases}
\begin{pmatrix}p_1^\theta & q_1^\theta \\ p_0^\theta & q_0^\theta \end{pmatrix} = \begin{pmatrix}a_0a_1+1 & a_1 \\ a_0 & 1 \end{pmatrix}\\
\begin{pmatrix}p_{n+1}^\theta & q_{n+1}^\theta \\ p_n^\theta & q_n^\theta \end{pmatrix} = \begin{pmatrix}a_{n+1} & 1 \\ 1 & 0 \end{pmatrix}\begin{pmatrix}p_n^\theta & q_n^\theta \\ p_{n-1}^\theta & q_{n-1}^\theta \end{pmatrix}\text{ for all $n\in \N\setminus\{0\}$.}
\end{cases}
\end{equation}
We then note that $\frac{p_n^\theta}{q_n^\theta}=[a_0, a_1, \ldots,a_n]$ for all $n \in \N$, and therefore $\left(\frac{p_n^\theta}{q_n^\theta}\right)_{n\in\N}$ converges to $\theta$ (see \cite{Hardy38}).

Expression (\ref{pq-rel-eq}) contains the crux for the construction of the Effros-Shen AF algebras. 

\begin{notation}
Throughout this paper, we shall employ the notation $x\oplus y \in X\oplus Y$ to mean that $x\in X$ and $y\in Y$ for any two vector spaces $X$ and $Y$ whenever no confusion may arise, as a slight yet convenient abuse of notation.
\end{notation}

\begin{notation}\label{af-theta-notation}
Let $\theta \in (0,1)\setminus\Q$ and $\theta = [a_j]_{j\in\N}$ be the continued fraction expansion of $\theta$. Let $(p_n^\theta)_{n\in\N}$ and $(q_n^\theta)_{n\in\N}$ be defined by Expression (\ref{pq-rel-eq}). We set $\af{\theta,0} = \C$ and, for all $n\in\N\setminus\{0\}$, we set:
\begin{equation*}
\af{\theta, n} = \alg{M}(q_{n}^\theta) \oplus \alg{M}(q_{n-1}^\theta) \text{,}
\end{equation*}
and:
\begin{equation*}
\alpha_{\theta,n} : a\oplus b \in \af{\theta,n} \longmapsto \begin{pmatrix}
a & & &  \\
  & \ddots & & \\
  &        & a & \\
  &        &   & b 
\end{pmatrix} \oplus a \in \af{\theta, n+1} \text{,}
\end{equation*}
where $a$ appears $a_{n+1}$ times on the diagonal of the right hand side matrix above. We also set $\alpha_0$ to be the unique unital *-morphism from $\C$ to $\af{\theta,1}$.

We thus define the Effros-Shen C*-algebra $\af{\theta}$, after \cite{Effros80b}:
\begin{equation*}
\af{\theta} = \varinjlim \left(\af{\theta,n}, \alpha_{\theta,n}\right)_{n\in\N}=\varinjlim \mathcal{I}_\theta\text{.}
\end{equation*} 
\end{notation}

Now, we display a new proof of \cite[Theorem 5.14]{AL} using the power of Theorem (\ref{af-cont}).
\begin{theorem}[{\cite[Theorem 5.14]{AL}}]\label{af-theta-thm-new}
Using Notation (\ref{af-theta-notation}) , the function:
\begin{equation*}
\theta \in ((0,1)\setminus \Q, \vert \cdot \vert) \longmapsto \left(\af{\theta}, \Lip_{\mathcal{I}_\theta ,\sigma_\theta}^{\beta_\theta} \right) \in (\mathcal{QQCMS}_{2,0}, \qpropinquity{} )
\end{equation*}
is continuous from $(0,1)\setminus\Q$, with its topology as a subset of $\R$, to the class of $(2,0)$-quasi-Leibniz quantum compact metric spaces metrized by the quantum propinquity $\qpropinquity{}$, where $\sigma_\theta$ is the unique faithful tracial state, and $\beta_\theta$ is the sequence of the reciprocal of dimensions of the inductive sequence, $\mathcal{I}_\theta$.
\end{theorem}
\begin{proof}
Let $(\theta^n)_{n \in \overline{N}} \subset (0,1)\setminus\Q$ such that $\lim_{n \to \infty} \theta^n =\theta^\infty$.  For each $n \in \overline{\N}$, let $\mathsf{cf}(\theta^n)=[a^{\theta^n}_j]_{j \in \N}$ denote the continued fraction expansion of $\theta^n$.  By \cite[Proposition 5.10]{AL}, the sequence $(\mathsf{cf}(\theta^n))_{n \in \N}$ converges to $\mathsf{cf}(\theta^\infty)$ in the Baire space metric defined in Definition (\ref{Baire-Space-def}).  By definition of convergence, there exists a non-decreasing sequence $(c_n)_{ n \in \N} \subset \N$ such that $\lim_{n \to \infty } c_n = \infty$, and if $k \geq c_N$, then  
\begin{equation*}\mathsf{d}_\BaireSpace ( \mathsf{cf}(\theta^{k}), \mathsf{cf}(\theta^\infty) )< \frac{1}{2^N}
\end{equation*} for each $N \in \N$.  By definition of the metric $\mathsf{d}_\BaireSpace$, this implies that for each $N \in \N$, if $k \in \N_{\geq c_N}$, then $a^{\theta^k}_n=a^{\theta^\infty}_n$ for all $n \in \{0, \ldots, N\}$ and thus the same holds for $p^{\theta^k}_n$ and $q^{\theta^k}_n$ by Equation (\ref{pq-rel-eq}).   Therefore by Notation (\ref{af-theta-notation}) and Definition (\ref{coll-def}), the family $\{\af{\theta^n} : n \in \overline{\N}\}$ is a fusing family with fusing sequence $(c_n)_{n \in \N}.$  Thus,  hypothesis (1) of Theorem (\ref{af-cont}) is satisfied.

For hypothesis (2) of Theorem (\ref{af-cont}), fix $N \in \N$ and assume $k \in \N_{\geq c_N}$. By \cite[Lemma 5.5]{AL} and Lemma (\ref{trace-conv-lemma}), we only need to show that $(t(\theta^k,N))_{n \in \N} \subset \R$ converges to $t(\theta^\infty,N)$, where $t(\theta, n)= (-1)^{n-1}q^\theta_n (\theta q^\theta_{n-1} - p^\theta_{n-1})$ for all $\theta \in (0,1) \setminus \Q$ and $n \in \N \setminus \{0\}$.  Now, by our fusing sequence $(c_n)_{n \in \N}$, if $k \geq c_N$, then   $t(\theta^k,N)=  (-1)^{N-1}q^{\theta^\infty}_N (\theta^k q^{\theta^\infty}_{N-1} - p^{\theta^\infty}_{N-1})$.  Therefore, since $\lim_{n \to \infty} \theta^n =\theta^\infty$, we have that $(t(\theta^k,N))_{n \in \N} \subset \R$ converges to $t(\theta^\infty,N)$, which establishes hypothesis (2) of Theorem (\ref{af-cont}).

For hypothesis (3) of Theorem (\ref{af-cont}), consider the continued fraction $\mathsf{cf}(\Phi)=[1]_{j \in \N}$, which is given by $\Phi=1-\phi$, where $\phi$ is the golden ratio $\phi=\frac{1+\sqrt{5}}{2}$.  By definition of the rational approximations defined by Equation (\ref{pq-rel-eq}), we have that $q^\theta_n \geq q^\phi_n$ for all $\theta \in (0,1)\setminus \Q$ and $n \in \N.$  Now, if $\beta_\theta (n) =\frac{1}{\dim (\af{\theta,n})}=\frac{1}{(q^\theta_n)^2 + (q^\theta_{n-1})^2} $ for all $\theta \in (0,1)\setminus \Q$ and $n \in \N \setminus 0$, then $\beta_\theta (n) \leq \beta_\Phi(n)$ for all $\theta \in (0,1)\setminus \Q$ and $n \in \N \setminus \{0\}$.   Therefore, the family of sequences  $\{\beta_{\theta^n} : n \in \overline{\N}\}$ along with the sequence $B(n)=\beta_\Phi(n)$ for all $n \in \N$ satisfy hypothesis (3) of Theorem (\ref{af-cont}) with the fusing sequence $(c_n)_{n \in \N}$.  
\end{proof}

\section{Leibniz Lip-norms for Unital AF algebras}\label{Leibniz-Lip-AF}

Our work in \cite{AL}  relied on the hypothesis of the existence of faithful tracial state for a unital AF algebra.  Of course, every simple unital AF algebra has a faithful tracial state, but in the non-simple case, there exist unital AF algebras without faithful tracial states. For example, consider the unitlization of the compact operators on an infinite-dimensional separable Hilbert space.  

To remedy this,  we introduce Leibniz Lip-norms that exist on any unital AF algebra built from quotient norms and the work of Rieffel in \cite{Rieffel11}, in which he established the Leibniz property for certain quotient norms.  Another consequence of this is that any unital AF algebra, $\A$, has finite dimensional approximations in propinquity provided by any  increasing sequence of unital finite dimensional subalgebras $(\A_n)_{n \in \N}$ such that $\A=\overline{\cup_{n \in \N} \A_n }^{\Vert \cdot \Vert_\A}$.  Furthermore, we show that any Lip-norm defined on the dense subspace $\cup_{n \in \N} \A_n $ proves this fact of finite-dimensional approximations in propinquity.  Therefore, we shall see that in some sense the C*-algebra structure of an AF algebra is enough to provide finite dimensional approximations in propinquity, which establishes a gratifying consistency between the C*-algebraic notion of approximately finite dimensional and the propinquity notion of approximately finite dimensional.

We note that the introduction of these Lip-norms from quotient norms in Theorem (\ref{AF-lip-norms-thm-best}) does not replace or diminish the importance of the Lip-norms from conditional expectations of Theorem (\ref{AF-lip-norms-thm}).  The conditional expectation Lip-norms give us explicit projections onto the C*-subalgebras, while also providing key estimates in quantum propinquity that  are crucial to our continuity results about AF algebras (see Theorem (\ref{af-theta-thm-new}) and Theorem (\ref{af-cont})).

We begin by providing some known examples of finite dimensional approximations in propinquity to gather a better understanding of the concept.  What is especially enlightening is that there are non-AF algebras that have natural finite-dimensional approximations in the sense of the propinquity.  We note for (1) of Example (\ref{af-propi-ex}) that for a compact metric space $X$, the C*-algebra $C(X)$ is AF if and only if $X$ is totally disconnected by \cite[Proposition 3.1]{Bratteli74}.
\begin{example}\label{af-propi-ex}
\begin{enumerate}
\item All C*-algebras of the form $C(X)$, where $(X, \mathsf{d}_X)$ is a compact metric space, have finite dimensional approximations in propinqtuiy induced by finite $\varepsilon$-nets, $X_\varepsilon \subseteq X$ and $C(X_\varepsilon)$. Indeed,  the Gromov-Hausdorff distance $\mathrm{GH}(X_\varepsilon , X) \leq \varepsilon $ by \cite[Example 7.3.11]{burago01} and \cite[Theorem 6.6]{Latremoliere13}  imply that $\qpropinquity{} \left(\left(C(X_\varepsilon) , \Lip_{\mathsf{d}_X}\right), \left(C(X) , \Lip_{\mathsf{d}_X}\right)\right) \leq \varepsilon.$

\item  Via a different approach than that of (1), the  commutative C*-algebra $C(S^2)$ has finite dimensional approximations in propinquity provided  by noncommutative finite-dimensional C*-algebras, where $S^2 = \{(x,y,z) \in \R^3 : x^2 + y^2 +z^2 =1 \}$.  Indeed, in \cite{Rieffel15},  finite dimensional approximations are provided by full matrix algebras.  

\item The quantum (noncommutative) tori  and $C\left(\T^2 \right)$, the C*-algebra of continuous functions on the torus  ---, which are non-AF, as presented in \cite{Latremoliere13c} have finite dimensional approximations in quantum propinquity provided  by fuzzy tori. 

\item Any unital AF algebra, $\A$ with faithful tracial state has finite dimensional approximations in propinquity provided  by any  inductive sequence of finite dimensional C*-algebras with inductive limit $\A$ by \cite{AL}.  
\end{enumerate}
\end{example}

One thing in common with all of these examples is that the existence of finite dimensional approximations in propinquity are proven using specific Lip-norms.  We shall see in Proposition (\ref{fd-approx-af-prop}) that in the case of unital AF algebras, the existence of a Lip-norm that is finite on the canonical dense subspace is all that is required to provided finite dimensional approximations in propinquity. In some sense, this means that the C*-algebra structure of an AF algebra is enough to provide finite dimensional approximations in propinquity.

\begin{notation}\label{lip-balls}
Let $(\A , \Lip_\A)$ be a quasi-Leibniz quantum compact metric space. Let $\mu \in \StateSpace (\A)$.  Denote:
\begin{equation*}
\begin{split}
& \Lipball{1}{\A, \Lip_\A}=\{ a \in \sa{\A} : \Lip_\A (a) \leq 1\}\\
& \Lipball{1}{\A, \Lip_\A, \mu }= \{ a \in \sa{\A} : \Lip_\A (a) \leq 1, \mu(a)=0\}.
\end{split}
\end{equation*}
\end{notation}
\begin{lemma}[{\cite{Latremoliere13}}]\label{haus-bd} Let $(\A , \Lip_\A) , (\B , \Lip_\B)$ be two quasi-Leibniz quantum compact metric spaces. If $\gamma=(\D, \omega , \pi_\A , \pi_\B)$ is a bridge of Definition (\ref{bridge-def}), then for any two states $\varphi_\A \in \StateSpace(\A), \varphi_\B \in \StateSpace(\B) $, we have that: 
\begin{equation*}
\begin{split}
&\Haus{\D} (\pi_\A(\Lipball{1}{\A, \Lip_\A})\omega, \omega \pi_\B (\Lipball{1}{\B, \Lip_\B})) \leq \\
&  \quad \Haus{\D} (\pi_\A(\Lipball{1}{\A, \Lip_\A, \varphi_\A })\omega, \omega \pi_\B (\Lipball{1}{\B, \Lip_\B , \varphi_\B })).
\end{split}
\end{equation*}
\end{lemma}

\begin{proof}
The proof is the argument in between \cite[Notation 3.13]{Latremoliere13} and \cite[Definition 3.14]{Latremoliere13}.
\end{proof}

To make for easier notation we will begin presenting results in the closure of the union case rather than the inductive limit case.  We shall see that this causes no issue in Proposition (\ref{ind-lim-union}).  We note that in the following result, the proof does not require any notion of quasi-Leibniz.  We  only include it to utilize the full power of the quantum propinquity. The proof of Proposition (\ref{fd-approx-af-prop}) also does not require that the subalgebras be finite-dimensional, but since no such example of a lip-norm exists yet of this form outside the AF case, we leave this assumption there. Also, thank you to F. \Latremoliere \ for pointing out \cite[Proposition 4.4]{Rieffel99} to me, which helped in the proof of  Proposition (\ref{fd-approx-af-prop}).  First, a remark on a crucial technique used in the proof of Proposition (\ref{fd-approx-af-prop}).
\begin{remark}
 Proposition (\ref{fd-approx-af-prop})  utilizes the notion of ``closing" a Lip-norm in a non-trivial way.  This notion was introduced by Rieffel in \cite{Rieffel99} in the comments preceding \cite[Proposition 4.4]{Rieffel99} to extend a Lip-norm onto the completion of a space.  Whereas, we use this notion to restrict our attention to a particular dense subspace to allow for finite-dimensional approximations.
\end{remark}
\begin{proposition}\label{fd-approx-af-prop}
Fix $C \geq 1, D\geq 0$.  Let $\A$ be a unital AF algebra such that $(\A,\Lip)$ is a $(C,D)$-quasi-Leibniz quantum compact metric space. Let  $(\A_n)_{n \in \N}$ be a sequence of unital finite dimensional C*-subalgebras of $\A$  such that $\A=\overline{\cup_{n \in \N} \A_n }^{\Vert \cdot \Vert_\A}$ and $\dom{\Lip}\supseteq \sa{\cup_{n \in \N} \A_n }$. Define a seminorm on $\sa{\A}$ by:
\begin{equation*}
\Lip_f (a)=
\begin{cases}
\Lip (a) &: \text{ if } a \in \sa{\cup_{n \in \N} \A_n}\\
\infty &: \text{ otherwise }. 
\end{cases}
\end{equation*} Let $L_{f,1}=\overline{\left\{ a \in \sa{\A}: \Lip_f (a) \leq 1\right\}}^{\Vert \cdot \Vert_\A}$. 

 If we let $\overline{\Lip}$ be the Minkowski functional of $L_{\A,1}$ on $\sa{\A}$, i.e.  
 \begin{equation*}\overline{\Lip}(a) = \inf \left\{ r>0 : \frac{1}{r}a \in L_{f,1} \right\}
 \end{equation*} for all $a \in \sa{\A}$,     then:
  \begin{equation}\label{minkowski-lip-ball-eq}
   \begin{split}  \ \Lip(a)= \Lip_f (a)=\Lip(a) < & \infty  \ \text{ for all }\ a \in \sa{\cup_{n \in \N} \A_n} \text{ and, }\\
    \left\{ a \in \sa{\A} : \overline{\Lip} (a) \leq 1 \right\} &= \overline{\left\{ a \in \sa{\cup_{n \in \N} \A_n} : \overline{\Lip} (a) \leq 1 \right\}}^{\Vert \cdot \Vert_\A}\\
  &= \overline{\left\{ a \in \sa{\cup_{n \in \N} \A_n} : \Lip (a) \leq 1 \right\}}^{\Vert \cdot \Vert_\A}, \end{split}
  \end{equation} and $(\A , \overline{\Lip})$ and $(\A_n , \overline{\Lip})=(\A_n, \Lip)$ are  $(C,D)$-quasi-Leibniz quantum compact metric space for all $n \in \N$ such that $
\lim_{n \to \infty} \qpropinquity{C,D} \left(\left(\A_n, \overline{\Lip} \right), \left(\A, \overline{\Lip}\right)\right) = 0.$
\end{proposition}
\begin{proof}
By \cite[Proposition IV.1.14]{Conway90}, the map $\overline{\Lip}$ is a  seminorm on $\sa{\A}$  such that  $\left\{ a \in \sa{\A} : \overline{\Lip} (a) \leq 1 \right\}=L_{f,1}$ since $ L_{f,1}$ is closed.  By construction, we have that $\Lip_f (a) \leq 1<\infty$ implies that $L_{f,1}= \overline{\left\{ a \in \sa{\cup_{n \in \N} \A_n} : \overline{\Lip} (a) \leq 1 \right\}}^{\Vert \cdot \Vert_\A}$.  The fact that $\overline{\Lip}(a)=\Lip_f(a)= \Lip (a)< \infty$ for all $ a \in \sa{\cup_{n \in \N} \A_n}$ is routine to check.  This establishes Expression (\ref{minkowski-lip-ball-eq}).  Also $\overline{\Lip}$ is a  lower semi-continuous seminorm with dense domain such that $\overline{\Lip}^{-1}(\{0\})=\R1_\A$ by Expression (\ref{minkowski-lip-ball-eq}).

 Next, we show that $(\A, \overline{\Lip})$ is a quantum compact metric space, and we use equivalence (3) of Theorem (\ref{Rieffel-thm}) to accomplish this. Let $q: a \in \sa{\A} \longmapsto a+\R1_\A \in \sa{\A}/\R1_\A$ denote the quotient map, which is continuous with respect to the assoicated norms. Now, since $(\A, \Lip)$ is a quantum compact metric space, we have that the set $q(\{a \in \sa{\A} : \Lip(a) \leq 1\})=\{ a+\R1_\A \in \sa{\A}/\R1_\A : \Lip(a) \leq 1 \}$ is totally bounded with respect to $\Vert \cdot \Vert_{\sa{\A}/\R1_\A}$ by Theorem (\ref{Rieffel-thm}).   Hence, by containment and Expression (\ref{minkowski-lip-ball-eq}), the set $q(\{ a \in \sa{\cup_{n \in \N} \A_n}: \overline{\Lip} (a) \leq 1\})=q(\{ a \in \sa{\cup_{n \in \N} \A_n}: \Lip (a) \leq 1\})$ is totally bounded with respect to the norm $\Vert \cdot \Vert_{\sa{\A}/\R1_\A}$. 
 
 Now, since $q$ is continuous, we have that 
 \begin{equation*}
 \begin{split}
  E&=
 q\left(\overline{\{ a \in \sa{\cup_{n \in \N} \A_n}: \overline{\Lip} (a) \leq 1\}}^{\Vert \cdot \Vert_\A}\right)\\
 &  \subseteq \overline{q(\{ a \in \sa{\cup_{n \in \N} \A_n}: \overline{\Lip} (a) \leq 1\})}^{\Vert \cdot \Vert_{\sa{\A}/\R1_\A}},
 \end{split}
 \end{equation*}  which implies that $E$ is totally bounded by containment and since the set on the right hand side is the closure of a totally bounded set.  By Expression (\ref{minkowski-lip-ball-eq}): 
 \begin{equation*}
 E=q(\{a \in \sa{\A} : \overline{\Lip}(a) \leq 1\})=\{a +\R1_\A \in \sa{\A}/\R1_\A:\overline{\Lip}(a) \leq 1 \}, 
 \end{equation*} and therefore, the pair $(\A, \overline{\Lip})$ is a  quantum compact metric space of Defintion (\ref{quasi-Monge-Kantorovich-def}) by Theorem (\ref{Rieffel-thm}).
 
     Now, we prove that $\overline{\Lip}$ is $(C,D)$-quasi Leibniz.  Our proof follows similarly to the proof of \cite[Proposition 3.1]{Rieffel10c}, which is the case of $C=1, D=0$.  Another similar result is \cite[Lemma 3.1]{Latremoliere13c}, which does involve a more general case than the quasi-Leibniz case. But,  we verify that these results still apply in our situation as there are some subtle differences with our construction with regard to the closedness of certain sets and conditions on the seminorm $\Lip$. 
     \begin{claim}
     The seminorm $\overline{\Lip}$ is $(C,D)$-quasi-Leibniz.
     \end{claim}
     \begin{proof}[Proof of claim] First, assume that $a, b \in \sa{\A}$ such that $\overline{\Lip}(a)=1=\overline{\Lip}(b)$. By definition of $\overline{\Lip}$,  there exists a sequence $(r_n)_{n \in \N} \subset \R$ that converges to $1=\overline{\Lip}(a)$ such that $r_n \geq 1=\overline{\Lip}(a)$ and $\frac{1}{r_n}a \in L_{f,1}$ for all $n \in \N$.  In particular, we have that the sequence $\left(\frac{1}{r_n}a\right)_{n \in \N} \subset L_{f,1}$ converges to $a$ with respect to $\Vert \cdot \Vert_\A$.  Now, by definition of $L_{f,1}$, for each $n \in \N$, choose $a_n \in \{c \in \sa{\cup_{n \in \N} \A_n} : \Lip(a) \leq 1\}$ such that $\left\Vert a_n -\frac{1}{r_n}a \right\Vert_\A < \frac{1}{n}$.  Therefore, for $n \in \N$, we have:
     \begin{equation*}
     \Vert a_n -a \Vert_\A \leq \left\Vert a_n -\frac{1}{r_n}a \right\Vert_\A +\left\Vert \frac{1}{r_n}a -a \right\Vert_\A < \frac{1}{n}+\left\Vert \frac{1}{r_n}a -a \right\Vert_\A.
     \end{equation*} Thus, the sequence $(a_n)_{n \in \N}$ converges to $a$ with respect to $\Vert \cdot \Vert_\A$ and $\Lip(a_n) \leq 1=\overline{\Lip}(a)$ for all $n \in \N$. Since  $a \neq 0_\A$ as $\overline{\Lip}(a)\neq 0$, up to dropping to a subsequence, we have that $\Vert a_n \Vert > 0$ for all $n \in \N$. Repeat the same process for $b$ to find a sequence $(b_n)_{n \in \N} \subseteq \sa{\cup_{n \in \N} \A_n}$ of non-zero terms such that $(b_n)_{n \in \N}$ converges to $b$ with respect to $\Vert \cdot \Vert_\A$, while $0<\Lip(b_n)\leq 1=\overline{\Lip}(b)$ for all $n \in \N$.  Now, for all $n \in \N$ we have that $a_nb_n +b_na_n \in \sa{\cup_{n \in \N}\A_n}$ such that $(a_nb_n +b_na_n )_{n \in \N}$ converges to $ab+ba$. Also, we gather that since $\Lip$ is quasi-Leibniz:
\begin{align*}
\Lip\left(\frac{a_nb_n +b_na_n}{2}\right)& \leq C(\Lip(a_n)\Vert b_n \Vert_\A + \Lip(b_n) \Vert a_n \Vert_\A )+D \Lip(a_n)\Lip(b_n)\\
& \leq C(\overline{\Lip}(a)\Vert b_n \Vert_\A + \overline{\Lip}(b)\Vert a_n \Vert_\A ) + D\overline{\Lip}(a)\overline{\Lip}(b)\\
& \leq C(\Vert b_n \Vert_\A + \Vert a_n \Vert_\A) + D,
\end{align*}
and since the right-hand side of the last inequality is non-zero, we have:  
\begin{equation*}
\Lip\left(\frac{a_nb_n +b_na_n}{2
\left(C(\Vert b_n \Vert_\A + \Vert a_n \Vert_\A) + D\right)}\right)\leq 1 \text{ for all } n \in \N.
\end{equation*}
Now, the sequence $\left(\frac{a_nb_n +b_na_n}{2
\left(C(\Vert b_n \Vert_\A + \Vert a_n \Vert_\A) + D\right)}\right)_{n \in \N}$ converges to $\frac{ab +ba}{2
\left(C(\Vert b \Vert_\A + \Vert a \Vert_\A) + D\right)}$ with respect to $\Vert \cdot \Vert_\A$ as all the scalars in the denominator are postive and converge to a positive scalar.  Thus, by Expression (\ref{minkowski-lip-ball-eq}), we have that:
\begin{align*}
\frac{ab +ba}{2
\left(C(\Vert b \Vert_\A + \Vert a \Vert_\A) + D\right)} & \in \overline{\left\{ c \in \sa{\cup_{n \in \N} \A_n} : \Lip(c)\leq 1\right\}}^{\Vert \cdot \Vert_\A}\\
&=\left\{ c\in \sa{\A} : \overline{\Lip}(c) \leq 1 \right\},
\end{align*}  
and thus:
\begin{equation}\label{quasi-lip-norm-1-eq}
\begin{split}
&\overline{\Lip}(a\circ b)=\overline{\Lip}\left(\frac{ab +ba}{2}\right) \leq C(\Vert b \Vert_\A + \Vert a \Vert_\A) + D \\
&  \text{ for all } a,b \in \sa{\A} \text{ such that } \overline{\Lip}(a)=1=\overline{\Lip}(b).
\end{split}
\end{equation}
The same holds true for the Lie product $\{a,b\}=\frac{ab-ba}{2i}$.  

Next, assume that $a,b \in \sa{\A}$ such that $\overline{\Lip}(a), \overline{\Lip}(b) \in (0, \infty) \subset \R$.  Hence, we have $\overline{\Lip}\left(\frac{1}{\overline{\Lip}(a)}a\right)=1=\overline{\Lip}\left(\frac{1}{\overline{\Lip}(b)}b\right).$ Thus:

\begin{align*}
\frac{1}{\overline{\Lip}(a)\overline{\Lip}(b)}\overline{\Lip}(a\circ b)&=\frac{1}{\overline{\Lip}(a)\overline{\Lip}(b)}\overline{\Lip}\left(\frac{ab+ba}{2}\right)\\
&= \overline{\Lip}\left( \frac{1}{\overline{\Lip}(a)}a \circ \frac{1}{\overline{\Lip}(b)}b\right) \\
&\leq C\left(\left\Vert \frac{1}{\overline{\Lip}(b)}b \right\Vert_\A + \left\Vert \frac{1}{\overline{\Lip}(a)}a \right\Vert_\A \right) + D\\
& = C\left(\frac{1}{\overline{\Lip}(b)}\Vert b \Vert_\A + \frac{1}{\overline{\Lip}(a)}\Vert a \Vert_\A \right) +D.
\end{align*}
 where Expression (\ref{quasi-lip-norm-1-eq}) was used in the last inequality.  Therefore:
 \begin{align*}
 \overline{\Lip}(a \circ b) & \leq \overline{\Lip}(a)\overline{\Lip}(b) \left( C\left(\frac{1}{\overline{\Lip}(b)}\Vert b \Vert_\A + \frac{1}{\overline{\Lip}(a)}\Vert a \Vert_\A \right) +D  \right)\\
 & =C \left( \overline{\Lip}(a) \Vert b \Vert_\A + \overline{\Lip}(b) \Vert a \Vert_\A \right) + D \overline{\Lip}(a)\overline{\Lip}(b)\\
 & \text{ for all } a,b \in \sa{\A} \text{ such that } \overline{\Lip}(a), \overline{\Lip}(b) \in (0, \infty),   
 \end{align*}
and the same holds for the Lie product $\{a,b\}$.  

Now, for $a, b\in \sa{\A}$, if either $\overline{\Lip}(a)=\infty$ or $\overline{\Lip}(b)=\infty$, then the conclusion is clear. Next, if $a,b \in \sa{\A}$ such that $\overline{\Lip}(a)=0=\overline{\Lip}(b)$, then $a,b \in \R1_\A$ and thus $a\circ b \in \R1_\A$ and $\overline{\Lip}(a \circ b)=0=\overline{\Lip}(\{a,b\})$, which concludes this case.   Finally, assume that $a,b \in \sa{\A}$,  $\overline{\Lip}(a)=0$ and $\overline{\Lip}(b) \in (0, \infty) \subset \R$.  Thus, $a=r1_\A$ for some $r \in \A$ and so $\Vert a\Vert_\A=\vert r\vert.$   Now: 
\begin{align*}\overline{\Lip}(a \circ b)&=   \overline{\Lip}\left(\frac{rb+rb}{2}\right)\\
&=\vert r \vert \overline{\Lip}(b)\\
&= \Vert a \Vert_\A \overline{\Lip}(b) \\
&\leq C\left( \Vert a \Vert_\A \overline{\Lip}(b) + \Vert b \Vert_\A \overline{\Lip}(a)\right) +D \overline{\Lip}(a)\overline{\Lip}(b)
\end{align*}
since $C \geq 1, D \geq 0$.
Also, we have  $\overline{\Lip}(\{a , b\})=   \overline{\Lip}\left(\frac{rb-rb}{2i}\right)=0$.  The same holds if the roles of $a$ and $b$ are switched.  Therefore, all cases are exhausted and the proof of the claim is complete
\end{proof}
Therefore, the pair $(\A, \overline{\Lip})$ is a $(C,D)$-quasi-Leibniz quantum compact metric space of Definition (\ref{quasi-Monge-Kantorovich-def}).  

For the C*-subalgebras. Fix $n \in \N.$  Since $\overline{\Lip}$ is a quasi-Leibniz Lip-norm defined on $\sa{\cup_{n \in \N} \A_n }$, it is routine to check that $\overline{\Lip}$ satisfies all but property (3) of Definition (\ref{quasi-Monge-Kantorovich-def}) on $\sa{\A_n}$.  To show  property (3), we begin by noting that  by Theorem (\ref{Rieffel-thm}) there exists some state $\mu \in \StateSpace(\A)$ such that the set $
 \Lipball{1}{\A, \overline{\Lip}, \mu }$ 
is totally bounded for $\Vert \cdot \Vert_\A$.  However, since the set:
\begin{equation*} \{ a \in \sa{\A_n} : \overline{\Lip} (a) \leq 1, \mu(a)=0\} \subseteq  \Lipball{1}{\A, \overline{\Lip}, \mu }
\end{equation*} and $\mu \in \StateSpace (\A_n)$, then by Theorem (\ref{Rieffel-thm}), the seminorm $\overline{\Lip}$ is a quasi-Leibniz Lip-norm on $\sa{\A_n} $.

Now, we prove convergence.  Let $ \varepsilon>0$. The fact that $\Lipball{1}{\A, \overline{\Lip} , \mu}$  is totally bounded by Theorem (\ref{Rieffel-thm}) implies that there exist $a_1 , \ldots , a_k \in  \Lipball{1}{\A, \overline{\Lip}, \mu }$ such that $
\Lipball{1}{\A, \overline{\Lip}, \mu} \subseteq \cup_{j=1}^k B_{\Vert \cdot \Vert_\A} \left(a_j , \varepsilon/3\right).$ By Expression (\ref{minkowski-lip-ball-eq}), for each $j \in \{1, \ldots , k\}$, there exist $a_j' \in \Lipball{1}{\A, \overline{\Lip}} \cap \sa{\cup_{n \in \N} \A_n}$ such that  $\Vert a_j-a_j' \Vert_\A < \varepsilon/3$,  and so $\vert \mu(a_j')\vert = \vert \mu(a_j-a_j')\vert \leq \Vert a_j-a_j' \Vert_\A < \varepsilon/3$
 since states are contractive by definition.  Hence: 
\begin{equation}\label{tot-bd}
\Lipball{1}{\A, \overline{\Lip}, \mu} \subseteq \cup_{j=1}^k B_{\Vert \cdot \Vert_\A} \left(a_j' , 2\varepsilon/3\right).
\end{equation}
Next, since $a_1' , \ldots , a_k' \in  \sa{\cup_{n \in \N} \A_n}$, let 
 $N=\min \{m  \in \N : \{a_1' , \ldots, a_k' \} \subseteq \A_m \}.$ 

 Fix $n \geq N$.  Let $a \in  \Lipball{1}{\A, \overline{\Lip} , \mu}$.  By Expression (\ref{tot-bd}), there exists $b \in \sa{\A_N} \subseteq \sa{\A_n }$ such that $b \in \Lipball{1}{\A_n, \overline{\Lip}}$, $\Vert a-b\Vert_\A < 2\varepsilon/3$, and $\vert \mu(b) \vert < \varepsilon/3$, where $\mu $ is seen as a state of $\A_n$.  Now, we have that $\mu(b) \in \R$  since $\mu$ is positive, and so $b-\mu(b)1_\A \in  \Lipball{1}{\A_n, \overline{\Lip},\mu}$ since Lip-norms vanish on scalars. Therefore: 
\begin{equation}\label{reach-ineq}
\left\Vert a-\left(b-\mu(b)1_\A\right)\right\Vert_\A \leq \Vert a-b\Vert_\A + \left\Vert \mu(b)1_\A \right\Vert_\A <  \varepsilon .
\end{equation}  In summary, for each $a \in \Lipball{1}{\A, \overline{\Lip} , \mu}$, there exists $c \in \Lipball{1}{\A_n, \overline{\Lip} , \mu}$ such that $\Vert a-c \Vert_\A < \varepsilon $ for $n \geq N$.
Now, if $a \in  \Lipball{1}{\A_n, \overline{\Lip}, \mu}$, then $a \in  \Lipball{1}{\A, \overline{\Lip} , \mu}$ and $\Vert a - a\Vert_\A =0< \varepsilon.$

Consider the bridge $\gamma=(\A , 1_\A , \mathrm{id}_\A , \iota_n )$ in the sense of Definition (\ref{bridge-def}),  where $\mathrm{id}_\A : \A \rightarrow \A$ is identity and $\iota_n : \A_n \rightarrow \A $ is inclusion.  But, since the pivot is $1_\A$, the height is $0$.  Now, combining Lemma (\ref{haus-bd}), Inequality (\ref{reach-ineq}) and the subsequent two sentences, we gather that the reach of the bridge is bounded by $ \varepsilon$.  Thus, by definition of length and Theorem-Definition (\ref{def-thm}), we conclude: 
\begin{equation*}
\qpropinquity{C,D} \left(\left(\A_n , \overline{\Lip} \right),\left(\A,\overline{\Lip} \right)\right) \leq \varepsilon, 
\end{equation*}
which establishes convergence.  The fact that $(\A_n, \Lip)=(\A_n, \overline{\Lip})$ for all $n \in \N$ is clear by Expression (\ref{minkowski-lip-ball-eq}), which completes the proof.
\end{proof}

In order for  Proposition (\ref{fd-approx-af-prop}) to have a powerful impact, we need to show that all unital AF algebras may be equipped with quasi-Leibniz Lip-norms. We show that this can be accomplished by using quotient norms and Rieffel's work on Leibniz seminorms and best approximations \cite{Rieffel11}. However,  we first prove a basic fact about certain Lip-norms in Proposition (\ref{complex-real-lip-prop}).  This fact is motivated by the observation that it can be the case that a candidate for a Lip-norm, $\Lip$, to be naturally defined on a unital dense subspace $\dom{\Lip}$ of $\A$ such that $\dom{\Lip} \cap \sa{\A}$ is  dense in $\sa{\A}$. Proposition (\ref{complex-real-lip-prop}) will allow us to verify a condition in this candidates natural setting of $\dom{\Lip}$ to induce a Lip-norm on $\dom{\Lip} \cap \sa{\A}$. An example of an application of Proposition (\ref{complex-real-lip-prop}) will be seen immediately in Theorem (\ref{AF-lip-norms-thm-best}).  But, first, a definition and some basic results about best approximations.

\begin{definition}
Let $\A$ be a Banach space with norm $\Vert \cdot \Vert_\A$.  We say that a norm closed subspace $\B \subseteq \A$ satisfies {\em best approximation } if for all $a \in \A$, there exists a $b_a \in \B$ such that  $\inf \{ \Vert a-b\Vert_\A : b \in \B \} = \Vert a-b_a\Vert_\A$, where $\Vert a \Vert_{\A/\B}= \inf \{ \Vert a-b\Vert_\A : b \in \B \} $ is the quotient norm.
\end{definition}
The following result is well-known.  However, we provide a proof.

\begin{lemma}\label{best-approx-sa-lemma}
Let $\A$ be a C*-algebra. If $\B \subseteq \A$ is a norm closed self-adjoint subspace of $\A$ that satisfies best approximation, then for all $a \in \sa{\A}$ there exists $b_a \in \sa{\B}$ such that the quotient norm $\Vert a \Vert_{\A/\B}= \Vert a-b_a \Vert_\A$.

Moreover, for all $a \in \sa{\A}$, the quotient norms $\Vert a \Vert_{\A/\B}=\Vert a \Vert_{\sa{\A}/\sa{\B}}$.
\end{lemma}
\begin{proof}
Let $a \in \sa{\A}$.  By assumption, there exists $b \in \B$ such that $\Vert a \Vert_{\A/\B} = \Vert a-b\Vert_\A.$  Now, set $b_a=\frac{b+b^*}{2} \in \sa{\B}$ and:

\begin{align*}
\left\Vert a-\frac{b+b^*}{2} \right\Vert_\A & = \left\Vert \frac{1}{2}a-\frac{1}{2}b+\frac{1}{2}a-\frac{1}{2}b^* \right\Vert_\A \\
& \leq \frac{1}{2} \Vert a-b \Vert_\A +\frac{1}{2} \left\Vert(a-b)^*\right\Vert_\A \\
& = \frac{1}{2} \Vert a-b \Vert_\A +\frac{1}{2} \left\Vert a-b\right\Vert_\A \\
& = \Vert a-b\Vert_\A  \\
&=\Vert a \Vert_{\A/\B}\\
&=\inf \{ \Vert a-c\Vert_\A : c \in \B\} \leq \left\Vert a-\frac{b+b^*}{2} \right\Vert_\A.
\end{align*}
Therefore, we gather that $\Vert a \Vert_{\A/\B}=\left\Vert a-\frac{b+b^*}{2} \right\Vert_\A=\Vert a-b_a \Vert_\A.$

Next, let $a \in \sa{\A}$, then by the above, there exists $b_a \in \sa{\B}$ such that:
\begin{align*}
\Vert a-b_a \Vert_\A &= \Vert a \Vert_{\A/\B}\\
&= \inf \{\Vert a-b\Vert_\A : b \in \B \} \\
& \leq \inf \{\Vert a-b\Vert_\A : b \in \sa{\B} \} \\
& = \Vert a \Vert_{\sa{\A}/\sa{\B}} \leq \Vert a-b_a \Vert_\A,
\end{align*}
which completes the proof.
\end{proof}

\begin{proposition}\label{complex-real-lip-prop}
Let $\A$ be a unital C*-algebra and let $\Lip$ be a seminorm defined on some dense unital subspace $\dom{\Lip}$ of $\A$ such that $\dom{\Lip}\cap \sa{\A}$ is a dense subspace of $\sa{\A}$ and $\{ a \in \dom{\Lip} : \Lip(a) = 0\}=\C1_\A$.

 If the set $\left\{a +\C1_\A \in \A/\C1_\A : a \in \dom{\Lip}, \Lip(a) \leq 1 \right\}$ is totally bounded in $\A/\C1_\A$ for $\Vert \cdot \Vert_{\A/\C1_\A}$, then 
 the pair $(\A, \Lip)$ formed by the dense unital  subspace $\dom{\Lip} \cap \sa{\A}$ of $\sa{\A}$   and the restriction of $\Lip$ to $\dom{\Lip} \cap \sa{\A}$ is a quantum compact metric space.
\end{proposition}
\begin{proof}
First, the set $\dom{\Lip} \cap \sa{\A}$ is a dense  subspace of $\sa{\A}$ and $\{ a \in \dom{\Lip} \cap \sa{\A}: \Lip(a) = 0 \} = \R1_\A$. Next, let $\left(a_n +\R1_\A \right)_{n \in \N} \subseteq \{a +\R1_\A \in \sa{\A}/\R1_\A : a \in \dom{\Lip}, \Lip(a) \leq 1\}$. The sequence $\left(a_n +\C1_\A\right)_{n \in \N} \subseteq \{a +\C1_\A \in \A/\C1_\A : a \in \dom{\Lip}, \Lip(a) \leq 1 \}$.  Hence, by assumption and total boundedness, there exists some Cauchy subsequence $\left(a_{n_k} +\C1_\A\right)_{k \in \N}$ with respect to $\Vert \cdot \Vert_{\A/\C1_\A}$.  

The space $\C1_\A$ is finite dimensional and therefore satisfies best approximation in $\A$.  Also, we have that $\sa{\C1_\A}=\R1_\A$.  Note that $a_n \in \sa{\A}$ for each $n \in \N$. Hence, by Lemma (\ref{best-approx-sa-lemma}) , the subsequence  $\left(a_{n_k} + \R1_\A\right)_{n\in \N}$ is Cauchy with respect to $\Vert \cdot \Vert_{\sa{\A}/\R1_\A}.$
\end{proof}

\begin{theorem}\label{AF-lip-norms-thm-best}
Let $\A$ be a unital AF algebra  with unit $1_\A$ such that $\mathcal{U} = (\A_n)_{n\in\N}$ is an increasing sequence of unital finite dimensional C*-subalgebras such that $\A=\overline{\cup_{n \in \N} \A_n}^{\Vert \cdot \Vert_\A}$, with $\A_0=\C 1_\A .$  For each $n \in \N$, we denote the quotient norm of $\A /\A_n $ with respect to $\Vert \cdot \Vert_\A$ by $S_n $.  Let $\beta: \N \rightarrow (0, \infty)$ have limit $0$ at infinity. 

If, for all $a \in  \A $, we set:
\begin{equation*}
\Lip_{\mathcal{U}}^\beta (a) =\sup \left\{ \frac{S_n (a)}{ \beta(n) } : n \in \N \right\},
\end{equation*}
then the domain of $\Lip_{\mathcal{U}}^\beta$ contains $\cup_{n \in \N} \A_n$, and 
\begin{enumerate}
\item using notation from Proposition (\ref{fd-approx-af-prop}), we have that $\left(\A, \Lip_{\mathcal{U}}^\beta \right)$, $\left(\A_n ,\Lip_{\mathcal{U}}^\beta \right)$, $\left(\A, \overline{\Lip_{\mathcal{U}}^\beta} \right)$,  and $\left(\A_n ,\overline{\Lip_{\mathcal{U}}^\beta} \right)$ for all $n \in \N$ are Leibniz quantum compact metric spaces where we view  $\Lip_{\mathcal{U}}^\beta$ restricted to $\sa{\A}$ such that
\item $
\lim_{ n \to \infty} \qpropinquity{1,0} \left(\left(\A_n ,\overline{\Lip_{\mathcal{U}}^\beta} \right),\left(\A, \overline{\Lip_{\mathcal{U}}^\beta} \right)\right)=0 
$ and 

$\lim_{ n \to \infty} \qpropinquity{1,0} \left(\left(\A_n ,\Lip_{\mathcal{U}}^\beta \right),\left(\A, \overline{\Lip_{\mathcal{U}}^\beta} \right)\right)=0.$ 
\end{enumerate}
\end{theorem}
\begin{proof}
We begin by proving (1). By \cite[Theorem 3.1]{Rieffel11},  for all $n \in \N$, we have that since $\A_n$ is unital, the quotient norm $S_n $ satisfies condition (2) of Definition (\ref{quasi-Monge-Kantorovich-def})  for $C=1$, $D=0$, and therefore,  so does $\Lip_{\mathcal{U}}^\beta$.  Thus $\Lip_{\mathcal{U}}^\beta$ is a Leibniz seminorm.  

To show that the seminorm only vanishes on scalars, note that  $\A_0=\C1_\A \subsetneq \A_n$ for each $n \in \N \setminus \{0\}$ implies that  ${\Lip_{\mathcal{U}}^\beta}^{-1}( \{ 0\})  = \C 1_\A $. 

For the domain, let $a \in \cup_{n \in \N} \A_n  $, then there exists $N \in \N$ such that $a \in \A_k$ for all $k \geq N$.  Therefore, the seminorm $S_k (a)=0$ for all $ k \geq N$, and hence, the seminorm $\Lip_{\mathcal{U}}^\beta$ evaluated at $a$ is a supremum over finitely many terms, and is thus finite.   Therefore, the domain of $\Lip_\mathcal{U}^\beta$ contains $\cup_{n \in \N} \A_n$.

 Since quotient norms are continuous, we have that $\Lip_\mathcal{U}^\beta$ is lower semi-continuous as a supremum of continuous real-valued maps.  

Now, note that since $\beta$ is convergent, we have $K=\sup \{\beta(n) : n \in \N\} < \infty$. Let $q_0: \A \rightarrow \A/\A_0=\A/\C1_\A$ denote the quotient map.  Define: \begin{equation*}\Lip_1 =\left\{ a \in \cup_{n \in \N} \A_n  : \Lip_{\mathcal{U}}^\beta(a) \leq 1\right\}.
\end{equation*}  By way of Proposition (\ref{complex-real-lip-prop}), we now show that $q_0(\Lip_1)$ totally bounded with respect to the quotient norm on $\A/\C1_\A $, in which the quotient norm is simply $S_0$ since $\A_0 =\C1_\A$.   Let $ \varepsilon >0 $.  By definition of $ \Lip_{\mathcal{U}}^\beta$, there exists $N \in \N$ such that  $\beta(N) < \varepsilon/3 $, so that $S_N(a) \leq \beta(N) < \varepsilon/3$ for all $ a\in \Lip_1 $.   Since $\A_N$ is a finite dimensional subspace, there exists a best approximation to $a$ in $\A_N$ for all $a\in \mathsf{L}_1$.   Thus, for all $a \in \Lip_1$, by axiom of choice, set $b_N(a) \in \A_N  $ to be one best approximation of $a$.  Define: 
\begin{equation*}
B_N = \{ b_N (a) \in \A_N  : a \in \Lip_1 \}.
\end{equation*}
If $a \in \Lip_1$, then since $\A_0=\C1_\A$ : 
\begin{align*}
S_0 (b_N(a)) & = 
 \inf \{\Vert b_N(a) - \lambda1_\A \Vert_\A: \lambda \in \C\}\\
 & = \inf \{ \Vert b_N (a) - a + a - \lambda 1_\A \Vert_\A : \lambda \in \C \} \\
& \leq \Vert b_N (a) - a \Vert_\A + \inf \{ \Vert a-\lambda 1_\A \Vert : \lambda \in \C \}\\
 & = S_N (a) +S_0 (a)\\
 & \leq \beta(N) +\beta(0) \leq 2K.
\end{align*}

Hence, the set $q_0(B_N)\subset \A_N /\C1_\A$ is bounded with respect $S_0$ on $\A_N/\C1_\A$, and therefore totally bounded with respect to $S_0$ on $\A_N/\C1_\A$ since $\A_N$ is finite dimensional. Let $F_N$ be a finite $\varepsilon/3$-net of $q_0 (B_N)$, so let $f_N=\{ b_N (a_1 ), \ldots , b_N (a_n)\in \A_N : a_j \in \Lip_1  , 1, \leq j \leq n< \infty \}$ such that $F_N=q_0(f_N)$.  

We claim that $q_0\left(\left\{a_1, \ldots, a_n\right\}\right)$ is  a finite $\varepsilon$-net for $q_0 (\Lip_1) $.  Indeed, let $a \in \Lip_1$, then $b_N (a) \in B_N$, so there exists $b_N(a_j) \in f_N$ such that $S_0 (b_N (a) - b_N (a_j))<  \varepsilon/3$.  Therefore: 
\begin{align*}
S_0 (a - a_j)& \leq S_0 (a-b_N(a) ) + S_0 (b_N (a) - b_N (a_j)) + S_0 (b_N (a_j ) - a_j ) \\
&  \leq \Vert a-b_N (a) \Vert_\A + \varepsilon/3 + \Vert b_N (a_j ) - a_j \Vert_\A \\
& = S_N (a) + \varepsilon/3 + S_N (a_j ) < \varepsilon.
\end{align*}
Hence, the set $q_0\left(\left\{a_1, \ldots, a_n\right\}\right)$ serves as a finite $\varepsilon$-net for $q_0 (\Lip_1) $.  Therefore,  by Proposition (\ref{complex-real-lip-prop}), the  
 pair $\left(\A, \Lip_{\mathcal{U}}^\beta \right)$ is a Leibniz quantum compact metric space, where we view  $\Lip_{\mathcal{U}}^\beta$ restricted to $\sa{\A}$. 

The remaining conclusions follow by Proposition (\ref{fd-approx-af-prop}).
\end{proof}

We note that (2) of Theorem (\ref{AF-lip-norms-thm-best}) is not obtained from an inequality like that of  (1), (2) of Theorem (\ref{AF-lip-norms-thm}), and we suspect that in general (2) of Theorem (\ref{AF-lip-norms-thm-best}) cannot be obtained from an inequality.  This is because it does not necessarily follow that for $ a\in \A, \Lip_\mathcal{U}^\beta(a) \leq 1$ we have that $\Lip_\mathcal{U}^\beta(b_n(a))\leq 1$ for any best approximation of $a$ in $\A_n$ for all $n \in \N$, which was a crucial step for the the inequality of Theorem (\ref{AF-lip-norms-thm}) achieved by conditional expectations rather than best approximations. This highlights a vital strength of the faithful tracial state case with the Lip-norms from Theorem (\ref{AF-lip-norms-thm}) since the inequality (1) of Theorem (\ref{AF-lip-norms-thm}) is crucial for our convergence results of AF algebras as seen in the proof of Theorem (\ref{af-cont}).  But, the Lip-norms of Theorem (\ref{AF-lip-norms-thm-best}) are vital for the general theory of AF algebras as quantum metric spaces to provide natural finite dimensional approximations  in propinquity for all unital AF algebras. 
\begin{remark}
Proposition (\ref{fd-approx-af-prop}) also provides Leibniz Lip-norms for any unital AF algebra from the Leibniz Lip-norms constructed in \cite[Theorem 2.1]{Antonescu04} for which the finite dimensional C*-algebras converge in the quantum propinquity.  Indeed, their construction of Lip-norms only require the existence of a faithful state, which exists on any  separable C*-algebra \cite[3.7.2]{Pedersen79}.
\end{remark}

\begin{remark}
 Proposition (\ref{fd-approx-af-prop}) and Theorem (\ref{AF-lip-norms-thm-best}) can be easliy translated to the inductive limit setting.
\end{remark}

\section{Quantum Isometries of AF algebras}

 We find sufficient conditions that provide quantum isometries (Theorem\\-Definition (\ref{def-thm})) between AF algebras with the Lip-norms from Theorem (\ref{AF-lip-norms-thm}), or when their distance is $0$ in the quantum propinquity, or when they produce the same equivlance classes that form the quantum propinqtuiy metric space. First, this is motivated by Bratteli's conditions for *-isomorphisms for AF algebras \cite[Theorem 2.7]{Bratteli72}.  Second, Inequality (\ref{limsup}) of Theorem (\ref{af-cont}) along with the convergence results of \cite{AL} display the importance of providing quantum isometries not only at the level of the entire AF algebra, but also at the level of the finite-dimensional C*-subalgebras.

We now present conditions for quantum isometries for AF algebras in the faithful tracial state case. We note that the hypotheses of the theorem are natural since they are chosen specifically to preserve the trace and the finite-dimensional structure of the AF algebra, which are the ingredients used to construct the Lip-norms. Also, Theorem (\ref{q-iso-trace-thm}) is used in \cite{Aguilar16} to find appropriate inductive limits that are quantum isometric to quotients. This is vital for convergence results since the inductive limit setting is more appropriate to provide convergence as seen in Section (\ref{conv-af})  and since most of our examples thus far are presented in the inductive limit setting.   But, first, we recast the Lip-norms from Theorem (\ref{AF-lip-norms-thm}) in the closure of union case instead of the inductive limit case to ease notation for the rest of the paper.   This causes no issues as presented in Proposition (\ref{ind-lim-union}).

\begin{theorem}\label{AF-lip-norms-thm-union} 
Let $\A$ be a unital AF algebra  with unit $1_\A$ endowed with a faithful tracial state $\mu$. Let $\mathcal{U} = (\A_n)_{n\in\N}$ be an increasing sequence of unital finite dimensional C*-subalgebras such that $\A=\overline{\cup_{n \in \N} \A_n}^{\Vert \cdot \Vert_\A}$ with $\A_0=\C 1_\A $.

Let $\pi$ be the GNS representation of $\A$ constructed from $\mu$ on the space $L^2(\A,\mu)$.

For all $n\in\N$, let:
\begin{equation*}
\CondExp{\cdot}{\A_n} : \A\rightarrow\A
\end{equation*}
be the unique conditional expectation of $\A$ onto $\A_n$ , and such that $\mu\circ\CondExp{\cdot}{\A_n} = \mu$.

Let $\beta: \N\rightarrow (0,\infty)$ have limit $0$ at infinity. If, for all $a\in\sa{\cup_{n \in \N} \A_n }$, we set:
\begin{equation*}
\Lip_{\mathcal{U},\mu}^\beta(a) = \sup\left\{\frac{\left\|a - \CondExp{a}{\A_n}\right\|_\A}{\beta(n)} : n \in \N \right\},
\end{equation*}
then $\left(\A,\Lip_{\mathcal{U},\mu}^\beta\right)$ is a {\Qqcms{2}}. Moreover, for all $n\in\N$:
\begin{equation*}
\qpropinquity{}\left(\left(\A_n,\Lip_{\mathcal{U},\mu}^\beta \right), \left(\A,\Lip_{\mathcal{U},\mu}^\beta \right)\right) \leq \beta(n)
\end{equation*}
and thus:
\begin{equation*}
\lim_{n\rightarrow\infty} \qpropinquity{}\left(\left(\A_n,\Lip_{\mathcal{U},\mu}^\beta \right), \left(\A,\Lip_{\mathcal{U},\mu}^\beta\right)\right) = 0\text{.}
\end{equation*}
\end{theorem}
\begin{proof}
The proof follows the same process of the proof of Theorem (\ref{AF-lip-norms-thm}) found in \cite[Theorem 3.5]{AL}.
\end{proof}
The Lip-norms of Theorem (\ref{AF-lip-norms-thm-union}) are compatible with the Lip-norms in the inductive limit case of Theorem (\ref{AF-lip-norms-thm}). The next Proposition (\ref{ind-lim-union}) establishes what we mean by this.

\begin{proposition}\label{ind-lim-union}
Let $\A$ be a unital AF algebra endowed with a faithful tracial state $\mu$. Let $\mathcal{I} = (\A_n,\alpha_n)_{n\in\N}$ is an inductive sequence of finite dimensional C*-algebras with C*-inductive limit $\A$, with $\A_0 \cong \C$ and where $\alpha_n$ is unital and injective for all $n\in\N$. Let $\beta: \N\rightarrow (0,\infty)$ have limit $0$ at infinity.

If we define $\mathcal{U} = (\indmor{\alpha}{n}(\A_n))_{n \in \N}$, then  the sequence $\mathcal{U}$ is an increasing sequence of unital finite dimensional C*-subalgebras of $\A$ such that $\A=\overline{\cup_{n \in \N} \indmor{\alpha}{n}(\A_n) }^{\Vert \cdot \Vert_\A}$ with $\indmor{\alpha}{0}(\A_0)=\C 1_\A $,   and the Lip-norms $\Lip_{\mathcal{I},\mu}^\beta=\Lip_{\mathcal{U},\mu}^\beta$, where $ \Lip_{\mathcal{I},\mu}^\beta$ is defined by Theorem (\ref{AF-lip-norms-thm}) and  $\Lip_{\mathcal{U},\mu}^\beta$ is defined by  (\ref{AF-lip-norms-thm-union}).\end{proposition}
\begin{proof}
By \cite[Chapter 6.1]{Murphy90}, the sequence $\mathcal{U}$ is an increasing sequence of unital finite dimensional C*-subalgebras of $\A$ such that $\A=\overline{\cup_{n \in \N} \indmor{\alpha}{n}(\A_n) }^{\Vert \cdot \Vert_\A}$ with $\indmor{\alpha}{0}(\A_0)=\C 1_\A $.  The equality of the Lip-norms $\Lip_{\mathcal{I},\mu}^\beta=\Lip_{\mathcal{U},\mu}^\beta$ follows by definition.
\end{proof}
Now, we study the faithful tracial state case.
\begin{theorem}\label{q-iso-trace-thm}
Let  $\A$ be a unital AF algebra  with unit $1_\A$ endowed with a faithful tracial state $\mu$. Let $\mathcal{U} = (\A_n)_{n\in\N}$ be an increasing sequence of unital finite dimensional C*-subalgebras such that $\A=\overline{\cup_{n \in \N} \A_n}^{\Vert \cdot \Vert_\A}$ with $\A_0=\C 1_\A $.  Let $\B$ be a unital AF algebra  with unit $1_\B$ endowed with a faithful tracial state $\nu$ and $\mathcal{V} = (\B_n)_{n\in\N}$ be an increasing sequence of unital finite dimensional C*-subalgebras such that $\B=\overline{\cup_{n \in \N} \B_n}^{\Vert \cdot \Vert_\B}$ with $\B_0=\C 1_\B .$ Let $\beta : \N \rightarrow (0,\infty)$ have limit $0$ at infinity.  Let $\Lip_{\mathcal{U}, \mu}^\beta, \Lip_{\mathcal{V}, \nu}^\beta$ denote the associated $(2,0)$-quasi-Leibniz Lip-norms from Theorem (\ref{AF-lip-norms-thm-union}) on $\A, \B$ respectively.

If $\phi : \A \hookrightarrow \B$ is a unital *-monomorphism such that the following hold:
\begin{enumerate}
\item $\phi(\A_n) =\B_n $ for all $n \in \N$, and
\item $\mu =\nu \circ \phi$,
\end{enumerate}
then: 
\begin{equation*}
\phi : \left(\A, \Lip_{\mathcal{U}, \mu}^\beta\right) \longrightarrow \left(\B,  \Lip_{\mathcal{V}, \nu}^\beta \right)
\end{equation*}
is a quantum isometry of Theorem-Definition (\ref{def-thm}) and: 
\begin{equation*}
\qpropinquity{2,0} \left(\left(\A, \Lip_{\mathcal{U}, \mu}^\beta\right), \left(\B,  \Lip_{\mathcal{V}, \nu}^\beta\right)\right)=0.
\end{equation*}
Moreover, for all $n \in \N$, we have:
\begin{equation*}
\qpropinquity{2,0} \left(\left(\A_n, \Lip_{\mathcal{U}, \mu}^\beta\right), \left(\B_n,  \Lip_{\mathcal{V}, \nu}^\beta\right)\right)=0.
\end{equation*}
\end{theorem}
\begin{proof}
Fix $ a \in \A $.  Let $n \in \N$.  Since $\B_n $ is finite dimensional, the C*-algebra $\B_n \cong \oplus_{j=1}^N \M(n(j))$ for some $N \in \N$ and $n(1), \ldots , n(N) \in \N \setminus \{0\}$ with *-isomorphism $\pi: \oplus_{j=1}^N \M(n(j)) \longrightarrow \B_n  $.   Let $E$ be the set of matrix units for  $\oplus_{j=1}^N \M(n(j))$ given in Notation (\ref{matrix-units}).  Define $E_\pi = \{ \pi (b) \in \B_n  : b \in E \}$. . Furthermore, since $\phi: \A \hookrightarrow \B$ is a *-monomorphism that satisfies hypothesis (1), the map $\phi: \A \rightarrow \B$ is a *-isomorphism by \cite[Theorem 2.7]{Bratteli72}. Hence,  by Proposition (\ref{cond-cont}) and  $\mu=\nu \circ \phi \iff \mu \circ \phi^{-1} =\nu$, we gather that: 
\begin{align*}
\quad \left\Vert \phi(a) - \CondExp{\phi(a)}{\B_n}\right\Vert_{ \B} 
& = \left\Vert \phi(a) -\sum_{e \in E_\pi} \frac{\nu \left(e^* \phi(a)\right)}{\nu  \left(e^*e \right)}e \right\Vert_{ \B}\\
&= \left\Vert \phi(a) -\sum_{e \in E_\pi} \frac{\mu \circ \phi^{-1} \left(e^* \phi(a)\right)}{\mu \circ \phi^{-1}  \left(e^*e \right)}e \right\Vert_{ \B}\\
&= \left\Vert \phi(a) -\sum_{e \in E_\pi} \frac{\mu  \left(\phi^{-1}(e^*)\phi^{-1}( \phi(a))\right)}{\mu \left( \phi^{-1}  \left(e^*e \right)\right)}e \right\Vert_{ \B}\\
&= \left\Vert \phi^{-1} \left(\phi(a) -\sum_{e \in E_\pi} \frac{\mu  \left(\phi^{-1}(e^*)a)\right)}{\mu \left( \phi^{-1}  \left(e^*e \right)\right)}e \right) \right\Vert_{ \A}\\
&= \left\Vert a -\sum_{e \in E_\pi} \frac{\mu  \left(\phi^{-1}(e^*)a)\right)}{\mu \left( \phi^{-1}  \left(e^*e \right)\right)}\phi^{-1}(e)  \right\Vert_{ \A}\\
&= \left\Vert a -\sum_{e' \in E} \frac{\mu  \left(\phi^{-1}\circ \pi (e'^*)a)\right)}{\mu \left( \phi^{-1} \circ \pi  \left(e'^*e' \right)\right)}\phi^{-1}\circ \pi(e')  \right\Vert_{ \A}\\
& = \left\Vert a - \CondExp{a}{\A_n}\right\Vert_{\A},
\end{align*}
where the last equality follows from  Proposition (\ref{cond-cont}) and the fact that $\phi^{-1} \circ \pi : \oplus_{j=1}^N \M(n(j)) \rightarrow \A_n$ is a *-isomorphism by assumption. 

Thus, since $n \in \N$ was arbitrary, we have:
\begin{equation}\label{lip-eq-eq}\Lip^\beta_{\mathcal{V},\nu }  \circ \phi (a) = \Lip^\beta_{\mathcal{U} ,\mu } (a)
\end{equation} for all $a\in \A$. Hence:
\begin{equation*}\phi : \left(\A, \Lip_{\mathcal{U}, \mu}^\beta\right) \longrightarrow \left(\B,  \Lip_{\mathcal{V}, \nu}^\beta \right)
\end{equation*} is a quantum isometry by Theorem-Definition (\ref{def-thm}).

Also, we have  $\left(\A_m , \Lip^\beta_{\mathcal{U} ,\mu }\right)$ is quantum isometric to 
$\left(\B_m , \Lip^\beta_{\mathcal{V},\nu }\right)$ by the map $\phi$ restricted to $\A_m$ for all $m \in \N$ by hypothesis (1), which completes the proof.
\end{proof}

Now, in Theorem (\ref{q-iso-best-thm}), we provide quantum isometries in the case of  the Leibniz Lip-norms from Theorem (\ref{AF-lip-norms-thm-best}) of the form $\Lip_{\mathcal{U}}^\beta$, and as a corollary, we will do the same for the Leibniz Lip-norms of the form $\overline{ \Lip_{\mathcal{U}}^\beta}$ with the same hypotheses.  Now, since neither of these Lip-norms require information about a faithful tracial state, the conditions to provide quantum isometries are weaker than for Theorem (\ref{q-iso-trace-thm}). Indeed:

\begin{theorem}\label{q-iso-best-thm}
 Let $\A$ be a unital AF algebra  with unit $1_\A$. Let $\mathcal{U} = (\A_n)_{n\in\N}$ be an increasing sequence of unital finite dimensional C*-subalgebras such that $\A=\overline{\cup_{n \in \N} \A_n}^{\Vert \cdot \Vert_\A}$ with $\A_0=\C 1_\A $.  Let $\B$ be a unital AF algebra  with unit $1_\B$  and $\mathcal{V} = (\B_n)_{n\in\N}$ be an increasing sequence of unital finite dimensional C*-subalgebras such that $\B=\overline{\cup_{n \in \N} \B_n}^{\Vert \cdot \Vert_\B}$ with $\B_0=\C 1_\B .$ Let $\beta : \N \rightarrow (0,\infty)$ have limit $0$ at infinity.  Let $\Lip_{\mathcal{U}}^\beta, \Lip_{\mathcal{V}}^\beta$ denote the associated Lip-norms from Theorem (\ref{AF-lip-norms-thm-best}) on $\A, \B$ respectively.

If $\phi : \A \hookrightarrow \B$ is a unital *-monomorphism such that $\phi(\A_n) =\B_n $ for all $n \in \N$, 
then  
\begin{equation*}
\phi : \left(\A, \Lip_{\mathcal{U}}^\beta\right) \longrightarrow \left(\B,  \Lip_{\mathcal{V}}^\beta \right)
\end{equation*}
is a quantum isometry of Theorem-Definition (\ref{def-thm}) and: 
\begin{equation*}
\qpropinquity{1,0} \left(\left(\A, \Lip_{\mathcal{U}}^\beta\right), \left(\B,  \Lip_{\mathcal{V}}^\beta\right)\right)=0.
\end{equation*}
Moreover, for all $n \in \N$, we have:
\begin{equation*}
\qpropinquity{1,0} \left(\left(\A_n, \Lip_{\mathcal{U}}^\beta\right), \left(\B_n,  \Lip_{\mathcal{V}}^\beta\right)\right)=0.
\end{equation*}
\end{theorem}
\begin{proof}
For each $n \in \N$, let $S^\A_n : \A/\A_n \longrightarrow \R$ denote the quotient norm and similarly denote $S^\B_n $. Fix $a \in \A$.  Since $\phi: \A \longrightarrow \B$ is a *-isomorphism by \cite[Theorem 2.7]{Bratteli72}, we have for all $n \in \N$: 
\begin{align*}
S^\B_n (\phi(a)) &= \inf \left\{ \Vert \phi(a) -b \Vert_\B : b \in \B_n \right\} \\
& = \inf \left\{ \Vert \phi^{-1}\left(\phi(a) -b \right)\Vert_\A : b \in \B_n \right\} \\ 
& = \inf \left\{ \Vert a -\phi^{-1}(b) \Vert_\A : b \in \B_n \right\} \\
& =  \inf \left\{ \Vert a -a'\Vert_\A : a' \in \A_n \right\}\\
& = S^\A_n (a), 
\end{align*}
where in the second to last equality we use the fact that $\phi^{-1}(\B_n) = \A_n $.  The rest of the proof follows exactly as the rest of the proof of Theorem (\ref{q-iso-trace-thm}) starting at Equation (\ref{lip-eq-eq}).
\end{proof}
We will now provide quantum isometries for the Lip-norms that provide the desirable convergence of finite-dimensional spaces as seen in Theorem (\ref{AF-lip-norms-thm-best}), and we see a direct application of the importance of having quantum isometries that preserve finite-dimensional approximations.  

\begin{corollary}\label{closed-q-iso-best-cor}
Let $\A$ be a unital AF algebra  with unit $1_\A$. Let $\mathcal{U} = (\A_n)_{n\in\N}$ be an increasing sequence of unital finite dimensional C*-subalgebras such that $\A=\overline{\cup_{n \in \N} \A_n}^{\Vert \cdot \Vert_\A}$ with $\A_0=\C 1_\A $.  Let $\B$ be a unital AF algebra  with unit $1_\B$  and $\mathcal{V} = (\B_n)_{n\in\N}$ be an increasing sequence of unital finite dimensional C*-subalgebras such that $\B=\overline{\cup_{n \in \N} \B_n}^{\Vert \cdot \Vert_\B}$ with $\B_0=\C 1_\B .$ Let $\beta : \N \rightarrow (0,\infty)$ have limit $0$ at infinity.  Let $\overline{\Lip_{\mathcal{U}}^\beta}, \overline{\Lip_{\mathcal{V}}^\beta}$ denote the associated Lip-norms from Theorem (\ref{AF-lip-norms-thm-best}) on $\A, \B$ respectively.

If $\phi : \A \hookrightarrow \B$ is a unital *-monomorphism such that $\phi(\A_n) =\B_n $ for all $n \in \N$, 
then there exists a quantum isometry  from $\left(\A, \overline{\Lip_{\mathcal{U}}^\beta} \right)$ to $\left(\B,  \overline{\Lip_{\mathcal{V}}^\beta}\right)$ and thus:
\begin{equation*}
\qpropinquity{1,0} \left(\left(\A, \overline{\Lip_{\mathcal{U}}^\beta} \right), \left(\B,  \overline{\Lip_{\mathcal{V}}^\beta}\right)\right)=0.
\end{equation*} 
Moreover: 
\begin{equation*}
\qpropinquity{1,0} \left(\left(\A_n, \overline{\Lip_{\mathcal{U}}^\beta}\right), \left(\B_n,  \overline{\Lip_{\mathcal{V}}^\beta}\right)\right)=0 \text{ for all } n \in \N.
\end{equation*}
\end{corollary}
\begin{proof}
By Proposition (\ref{fd-approx-af-prop}), we have  that $\left(\A_n, \overline{\Lip_{\mathcal{U}}^\beta}\right)=\left(\A_n,\Lip_{\mathcal{U}}^\beta\right)$ and  $\left(\B_n,  \overline{\Lip_{\mathcal{V}}^\beta}\right)= \left(\B_n,  \Lip_{\mathcal{V}}^\beta \right)$ for all $n \in \N$.  Thus, by  Theorem (\ref{q-iso-best-thm}), we have that: 
\begin{equation*} \qpropinquity{1,0} \left(\left(\A_n, \overline{\Lip_{\mathcal{U}}^\beta}\right), \left(\B_n,  \overline{\Lip_{\mathcal{V}}^\beta}\right)\right)=0
\end{equation*}  for all  $ n \in \N$ by the triangle inequality.

Next, by  the triangle inequality, we have:
\begin{align*}
\qpropinquity{1,0} \left(\left(\A, \overline{\Lip_{\mathcal{U}}^\beta} \right), \left(\B,  \overline{\Lip_{\mathcal{V}}^\beta}\right)\right)& \leq  \qpropinquity{1,0} \left(\left(\A, \overline{\Lip_{\mathcal{U}}^\beta} \right), \left(\A_n,  \overline{\Lip_{\mathcal{U}}^\beta}\right)\right)\\
& \quad  + \qpropinquity{1,0} \left(\left(\A_n, \overline{\Lip_{\mathcal{U}}^\beta} \right), \left(\B_n,  \overline{\Lip_{\mathcal{V}}^\beta}\right)\right) \\
& \quad + \qpropinquity{1,0} \left(\left(\B_n, \overline{\Lip_{\mathcal{U}}^\beta} \right), \left(\B,  \overline{\Lip_{\mathcal{V}}^\beta}\right)\right)\\
&=\qpropinquity{1,0} \left(\left(\A, \overline{\Lip_{\mathcal{U}}^\beta} \right), \left(\A_n,  \overline{\Lip_{\mathcal{U}}^\beta}\right)\right)  + 0 \\
& \quad + \qpropinquity{1,0} \left(\left(\B_n, \overline{\Lip_{\mathcal{U}}^\beta} \right), \left(\B,  \overline{\Lip_{\mathcal{V}}^\beta}\right)\right).
\end{align*}
Hence, we have $\qpropinquity{1,0} \left(\left(\A, \overline{\Lip_{\mathcal{U}}^\beta} \right), \left(\B,  \overline{\Lip_{\mathcal{V}}^\beta}\right)\right)=0$ by part (2) of Theorem (\ref{AF-lip-norms-thm-best}).  Therefore, by Theorem-Defintion (\ref{def-thm}), there exists a quantum isometry from $\left(\A, \overline{\Lip_{\mathcal{U}}^\beta} \right)$ to $\left(\B,  \overline{\Lip_{\mathcal{V}}^\beta}\right)$.
\end{proof}

\bibliographystyle{plain}
\bibliography{thesis-a}

\vfill

\end{document}